\documentclass[12pt, a4paper, reqno]{amsart}
\usepackage{amscd, amssymb}
\usepackage{amsthm}
\usepackage{tikz}
\usetikzlibrary{arrows,positioning} 
\usepackage{mathrsfs}
\usepackage{bbm}

\usepackage{multirow}
\usepackage{makecell}

\usepackage{hyperref}
\hypersetup{
	colorlinks=true,
	linkcolor=blue,
	filecolor=magenta,  
	urlcolor=cyan,
}

\usepackage{hyperref}
\usepackage[capitalise]{cleveref}

\def\Aut{\operatorname{Aut}}
\def\End{\operatorname{End}}

\def\dim{\operatorname{dim}}

\def\id{\operatorname{id}}
\def\Dist{\operatorname{Dist}}

\def\Ind{\operatorname{Ind}}

\def\Ad{\operatorname{Ad}}

\def\C{\mathbb{C}}

\def\N{\mathbb{N}}

\def\G{\mathbb{G}}
\def\AA{\mathcal{A}}

\def\UU{\mathcal{U}}

\def\DD{\mathcal{D}}

\def\PP{\mathcal{P}}

\def\ZZ{\mathcal{Z}}

\def\a{\mathfrak{a}}
\def\b{\mathfrak{b}}
\def\c{\mathfrak{c}}
\def\m{\mathfrak{m}}
\def\n{\mathfrak{n}}
\def\p{\mathfrak{p}}

\def\g{\mathfrak{g}}
\def\h{\mathfrak{h}}

\def\k{\mathfrak{k}}
\def\l{\mathfrak{l}}
\def\s{\mathfrak{s}}
\def\o{\mathfrak{o}}

\def\d{\partial}

\def\ol{\overline}

\def\sub{\subseteq}

\def\xto{\xrightarrow}

\newtheorem{thm}{Theorem}[section]
\newtheorem{cor}[thm]{Corollary}
\newtheorem{lemma}[thm]{Lemma}
\newtheorem{prop}[thm]{Proposition}
\newtheorem{conj}[thm]{Conjecture}

\theoremstyle{definition}
\newtheorem{definition}[thm]{Definition}
\theoremstyle{remark}
\newtheorem{remark}[thm]{Remark}

\numberwithin{equation}{section}

\usepackage{relsize}
\usepackage{fancyhdr}
\usepackage[all,cmtip]{xy}

\begin{document}
	\title{Ghost distributions on supersymmetric spaces II: basic classical superalgebras}
	
	\author[Alexander Sherman]{Alexander Sherman}
	
	\maketitle
	\pagestyle{plain}

	\begin{abstract} 
		We study ghost distributions on supersymmetric spaces for the case of basic classical Lie superalgebras.  We introduce the notion of interlaced pairs, which are those for which both $(\g,\k)$ and $(\g,\k')$ admit Iwasawa decompositions.  For such pairs we define a ghost algebra, generalizing the subalgebra of $\UU\g$ defined by Gorelik.  We realize this algebra as an algebra of $G$-equivariant operators on the supersymmetric space itself, and for certain pairs, the `special' ones, we realize our operators as twisted-equivariant differential operators on $G/K$.  We additionally show that the Harish-Chandra morphism is injective, compute its image for all rank one pairs, and provide a conjecture for the image when $(\g,\k)$ is interlaced.
	\end{abstract}

\section{Introduction}

Let $\g$ be one of the basic classical Lie superalgebras $\g\l(m|n)$, $\o\s\p(m|2n)$, $\mathfrak{d}(1|2;\alpha)$, $\g(1|2)$, or $\a\b(1|3)$.  Fix a nondegenerate invariant supersymmetric form $(-,-)$ on $\g$ and let $\theta$ be an involution of $\g$ which preserves $(-,-)$.  Write $\g=\k\oplus\p$ for the $(\pm1)$-eigenspace decomposition with respect to $\theta$, and set $\k'=\k_{\ol{0}}\oplus\p_{\ol{1}}$; in particular $\k'$ is the fixed points of the involution $\delta\circ\theta$, where $\delta(v)=(-1)^{\ol{v}}v$ is the canonical grading automorphism.   Let $G$ be an algebraic supergroup with $\operatorname{Lie}G=\g$ and $G_0$ reductive, and suppose that $\theta$ extends to an involution of $G$ and that $K$ is a subgroup of $G$ with $(G^\theta)^\circ\sub K\sub G^{\theta}$.  Here $(G^\theta)^\circ$ denotes the identity component of $G^\theta$.  Finally we set $K'$ to be the subgroup of $G$ with $K'_0:=K_0$ and $\operatorname{Lie}K':=\k'$.

In \cite{sherman2021ghost} it was shown that the action of $K'$ on $G/K$ enjoys many nice properties, with the foremost being that $\Dist(G/K,eK)=\UU\g/\UU\g\k$ is isomorphic, as a $K'$-module, to $\Ind_{\k_{\ol{0}}}^{\k'}\Dist(G_0/K_0,eK_0)$, where algebraically $\Dist(G_0/K_0,eK_0)=\UU\g_{\ol{0}}/\UU\g_{\ol{0}}\k_{\ol{0}}$. We set $\AA_{G/K}$ to be the $K'$-invariant distributions on $G/K$ supported at $eK$, which we call the space of ghost distributions of $G/K$. Then there is an explicit isomorphism of vector spaces $\Dist(G_0/K_0,eK_0)^{K_0}\to\AA_{G/K}$.  

Write $\a\sub\p_{\ol{0}}$ for a Cartan subspace, and let us suppose that we have an Iwasawa decomposition $\g=\k\oplus\a\oplus\n$; then we may define the Harish-Chandra morphism \newline $HC:\AA_{G/K}\to S(\a)$.  The collection of polynomials $HC(\AA_{G/K})\sub S(\a)$ form a module over $HC(\ZZ_{G/K})$, and carry representation-theoretic information about branching from $G$ to $K'$; in particular they detect when an injective indecomposable on the trivial module of $K'$ appears in highest weight $G$-submodules of $G/K$ (see \cref{lemma branching}).  The first important result we have is:
\begin{thm}\label{intro thm inj}
	The map $HC:\AA_{G/K}\to S(\a)$ is injective.
\end{thm}
We can also compute the degree and highest order term of $HC(\gamma)$ for $\gamma\in\AA_{G/K}$.  For representation-theoretic consequences of this result, see \cref{cor generic typicality}.  Note that \cref{intro thm inj} relies on the existence of a form which $\theta$ preserves, and that without this assumption the statement is not true.  

\subsection{Interlaced pairs} To further our study of $HC(\AA_{G/K})$, we show that certain pairs $(\g,\k)$ have a very close relationship to $(\g,\k')$.  Namely, we say that $(\g,\k)$ is interlaced if there exists $t\in A$ (where $A=\exp(\a)\sub GL(\g)$) such that if we set $\phi=\Ad(t)$ then $\phi\theta=\delta\theta\phi$, and further $\phi^2=\delta$ ($\delta$ is the grading operator).  This in particular implies that $\phi(\k)=\k'$ and $\phi(\k')=\k$.  The following remarkable result is crucial for many ideas in our work.
\begin{thm}\label{intro thm interlacing}
	A supersymmetric pair $(\g,\k)$ is interlaced if and only if both $(\g,\k)$ and $(\g,\k')$ admit Iwasawa decompositions.
\end{thm}
Note that the forward direction is clear; for the backward direction we use the classification of supersymmetric pairs of the type we consider along with properties of their reduced root systems.

\subsection{Ghost centre}\label{sec_intro_ghost} The property of a pair being interlaced gives it structure that is similar to the diagonal pair $(\g\times\g,\g)$.  Write $\ZZ_{G/K}$ for the $K$-invariant distributions on $\Dist(G/K,eK)$. Then the space $\widetilde{\ZZ_{G/K}}:=\ZZ_{G/K}+\AA_{G/K}$ may be endowed with the natural structure of an algebra, which we call the ghost algebra of  $G/K$, and it is such that the map $HC:\widetilde{\ZZ_{G/K}}\to S(\a)$ is an algebra homomorphism.  In particular, we obtain in this case that 
\[
HC(\AA_{G/K})HC(\AA_{G/K})\sub HC(\ZZ_{G/K}).
\]
Note that in \cite{sherman2021ghost} it was shown that $HC(\AA_{G/K})$ is always a module over $HC(\ZZ_{G/K})$, even when $(\g,\k)$ is not interlaced. 

Remarkably, we may lift $\widetilde{\ZZ_{G/K}}$ to an algebra of $G$-equivariant operators on $G/K$:
\begin{thm}\label{intro thm lifting to ops}
	Suppose that $(\g,\k)$ is interlaced with interlacing automorphism $\phi=\Ad(t)$, and assume that $t^2\in K(\Bbbk)$.  Then we have an injective morphism of algebras
	\[
\widetilde{\ZZ_{G/K}}\to\End_{G}(\Bbbk[G]), \ \ \ \ u\mapsto \tilde{u}.
	\]
\end{thm}
See \cref{sec lifting ops} for the definition of $\tilde{u}$ when $u\in\AA_{G/K}$.  Note that $\tilde{u}$ will \emph{not} be a differential operator when $u\notin\ZZ_{G/K}$.   However for certain interlaced pairs this issue can be remedied, namely when $t$ can be chosen to be central in $G_0$, as we see next.

\subsection{Special pairs} We call a supersymmetric pair $(\g,\k)$ `special' if it is interlaced, and the interlacing element $t$ can taken to be central in $G_0$.  The special pairs are exactly $(\g\l(m|2n),\o\s\p(m|2n))$ and $(\o\s\p(2|2n),\o\s\p(1|2r)\times\o\s\p(1|2n-2r))$.  In this case, for $u\in\AA_{G/K}$, we may set 
\[
D_u:=L_{t^{-1}}^*\circ R_t^*\circ(1\otimes u)\circ a^*,
\]
where $R_t,L_{t^{-1}}$ are respectively left and right translation by $t$ on $G$.  Then in Section 7 it is shown that $D_u$ is an $\Ad(t)$-twisted equivariant differential operator; in particular it is $G_0$-equivariant.  In fact, there we obtain a subalgebra $\DD^{G,\cdot}(G/K)\sub \DD(G/K)$ consisting of differential operators which are $\Ad(z)$ twisted equivariant for some $z$ in the center of $G_0$.  Thus we obtain an analogue of the full ghost centre as introduced in \cite{sherman2021ghost}.

The action of this algebra on $\C[G/K]$, and in particular its eigenvalues on highest weight functions, would be very interesting to understand.  Because they are differential operators, there is some hope that the methods of \cite{alldridge2012harish} may apply, for instance, to help prove Weyl group invariance.

\subsection{$HC(\AA_{(\g,\k)})$ for rank one pairs} In the last section we compute $HC(\AA_{(\g,\k)})$ for all rank one pairs.  Here $\AA_{(\g,\k)}=(\UU\g/\UU\g\k)^{\k'}$.  Here is a list of our results; we use standard notation for root systems, and in particular follow the conventions from \cite{sherman2020iwasawa}.
\begin{enumerate}
	\item[(i)] $(\g\l(m|n),\g\l(m-1|n)\times\g\l(1))$: let $t=\frac{1}{2}h_{\epsilon_1-\epsilon_m}$; then
	\[
	HC(\AA_{(\g,\k)})= \Bbbk[t(t-n+m-1)]\langle t(t-1)\cdots(t-(n-1))\rangle.
	\]
	\item[(ii)] $(\o\s\p(2|2n),\o\s\p(1|2n))$: let $t=h_{\epsilon_1}$, where $\epsilon_1(h_{\epsilon_1})=1$; then 
	\[
	HC(\AA_{(\g,\k)})=\{p\in \Bbbk[t]:p(n+r)=(-1)^{r}p(n-r):1\leq r\leq n\},
	\]
	or more explicitly:
	\[
	\Bbbk[t(t-2n)]\langle (t-1)(t-3)\cdots(t-(2n-1)),t(t-2)\cdots(t-2n)\rangle,
	\]
	\item[(iii)] $(\o\s\p(m|2n),\o\s\p(m-1|2n))$, $m\geq3$: let $t=h_{\epsilon_1}$; then
	\[
	HC(\AA_{(\g,\k)})= \Bbbk[t(t-2n+m-2)]\langle(t-1)(t-3)\cdots(t-(2n-1))\rangle.
	\]
	\item[(iv)] $(\o\s\p(m|2n),\o\s\p(m|2n-2)\times\s\p(2))$, $n\geq 2$: let $t\in\a$ be such that $(\delta_1+\delta_2)(t)=1$; then
	\[
	HC(\AA_{(\g,\k)})= \Bbbk[t(t+2n-m-1)]\langle (t+1)t(t-1)\cdots(t-(m-2))\rangle
	\]
\end{enumerate}

We observe that if $(\g,\k)$ is not interlaced, then we do not have any kind of Weyl group (anti-)invariance.  The conditions determining $HC(\AA_{(\g,\k)})$ in this general case remain mysterious.  However for interlaced pairs we have a coherent conjecture, as is explained in what follows.

\subsection{Conjecture for interlaced pairs}\label{sec_intro_conj} Let $\ol{\Delta}$ denote the reduced root system of $(\g,\k)$ with respect to a chosen Cartan subspace $\a$, and set $\ol{\rho}$ to be the restricted Weyl vector.  Let $\ol{\Delta}_{ev}=\ol{\Delta_{\ol{0}}}\setminus\{0\}$, that is the non-zero projections of even roots to $\a^*$.  Set $\ol{\Delta}_{odd}:=\ol{\Delta}\setminus\ol{\Delta}_{ev}$.  For $\alpha\in\ol{\Delta}_{ev}$, write $r_{\alpha}$ for the reflection on $\a^*$ determined by $\alpha$.  Let $W$ denote the subgroup of $GL(\a)$ generated by $r_{\alpha}$, for all $\alpha\in\ol{\Delta}_{ev}$. For $p\in S(\a)$ and $w\in W$, we write $(w.p)(\lambda)=p(w(\lambda+\ol{\rho})-\ol{\rho})$.


\begin{conj}\label{conj_intro}
	Assume that $(\g,\k)$ is interlaced. Then $HC(\AA_{(\g,\k)})$ is given by the set of $p\in S(\a)$ satisfying:
	\begin{enumerate}
		\item[(i)] for $\alpha\in\ol{\Delta}_{ev}$ with $n_{\alpha}=\frac{1}{2}\dim(\g_{\alpha})_{\ol{1}}$, we have
		 \[
		(r_{\alpha}.p)=(-1)^{n_{\alpha}}p;
		\]
		\item[(ii)] for $\alpha\in\ol{\Delta}_{odd}$, and $(\lambda+\ol{\rho},\alpha)=0$, we have
		\[
		p(\lambda+r\alpha)=(-1)^{r}p(\lambda-r\alpha)
		\]
		for $1\leq r\leq \frac{1}{2}\dim(\g_{\alpha})_{\ol{1}}$.
	\end{enumerate}
\end{conj}

\begin{remark}
	We first observe that the above conjecture is valid for all group-like cases $(\g\times\g,\g)$, as shown in \cite{gorelik2000ghost}.  All interlaced cases that we have computed in this paper also satisfy the above conjecture.  Further, observe that if $\alpha\in\ol{\Delta}_{odd}$ and $(\alpha,\alpha)=0$, then in fact the second condition implies that
	\[
(h_{\alpha}+(\ol{\rho},\alpha))|HC(D)
	\]
	for all $D\in\AA_{(\g,\k)}$.  (See \cref{iso_root_vanishing_general} for the case when $\alpha$ is simple)
\end{remark}

\begin{remark}
	From \cite{alldridge2012harish}, one may deduce the following description of $HC(\ZZ_{(\g,\k)})$: it is given by the set of $p\in S(\a)$ such that 
	\begin{enumerate}
		\item[(i)] $w. p=p$ for all $w\in W$;
	\item[(ii)] for $\alpha\in\ol{\Delta}_{\ol{1}}$, and $(\lambda+\ol{\rho},\alpha)=0$, we have
\[
p(\lambda+r\alpha)=p(\lambda-r\alpha)
\]
for $1\leq r\leq \frac{1}{2}\dim(\g_{\alpha})_{\ol{1}}$.
	\end{enumerate}
Thus our conditions are a `square-root' of the conditions on $HC(\ZZ_{(\g,\k)})$, which is a necessity for interlaced pairs since $HC(\AA_{(\g,\k)})^2\sub HC(\ZZ_{(\g,\k)})$ as explained in \cref{sec_intro_ghost}.
\end{remark}

\subsection{Future directions} In future work, we look to prove \cref{conj_intro}, and understand what occurs in the cases of non-interlaced pairs.  Further, for the special pairs we want to compute the structure of the algebra of twisted-equivariant differential operators.  In general, the structure of the ghost centre $\widetilde{\ZZ_{(\g,\k)}}$ for interlaced pairs is not understood, and should be determined.
  
\subsection{Summary of sections}  In section 2 we recall the setup and results of \cite{sherman2021ghost}.  Section 3 is devoted to the proof of the injectivity of the Harish-Chandra morphism and its consequences.  Section 4 defines the notion of interlaced pairs, and proves \cref{intro thm interlacing}.  Section 5 constructs the ghost algebra of $G/K$ for interlaced pairs, and Section 6 explains how to lift to equivariant operators in this case.  Section 7 looks at special supersymmetric pairs, showing we may construct a `full' ghost algebra, and explains how to lift these invariant distributions to twisted-equivariant differential operators, giving an algebra of differential operators on $G/K$.  Section 8 discusses two tools which may be used to help compute $HC(\AA_{G/K})$.  Finally Section 9 computes $HC(\AA_{G/K})$ in every rank one.

\subsection{Acknowledgements}  The author is grateful to Alexander Alldridge, Maria Gorelik, Thorsten Heidersdrof, Shifra Reif, and Vera Serganova for numerous helpful discussions.   This research was partially supported by ISF grant 711/18 and NSF-BSF grant 2019694.

\section{Recollections}

In what follows, $\Bbbk$ denotes an algebraically closed field of characteristic zero.

We work with the same notation and setup as in \cite{sherman2021ghost} except that we restrict to certain Lie supergroups and Lie superalgebras.  In particular, unless stated otherwise, $\g$ will always denote one of the Lie superalgebras $\g\l(m|n)$, $\o\s\p(m|2n)$, $\mathfrak{d}(1,2;a)$, $\g(1,2)$, or $\a\b(1,3)$.   Thus $\g$ is quasireductive and admits a nondegenerate, invariant supersymmetric bilinear form which we denote by $(-,-)$.   For more on the basic properties of these Lie superalgebras we refer to \cite{musson2012lie}.  Further, $G$ will denote a quasireductive Lie supergroup with $\operatorname{Lie}G=\g$.  For more on quasireductive Lie supergroups we refer to \cite{serganova2011quasireductive}.

We are again interested in supersymmetric pairs.  Throughout, $\theta$ will denote an involution of $\g$ which preserves the form $(-,-)$, with fixed point subalgebra $\k$ and $(-1)$-eigenspace $\p$.  In particular $\k$ will be quasireductive and itself admits a nondegenerate, invariant supersymmetric form.  We have an explicit classification of the pairs $(\g,\k)$ that we consider, given in Section 4.  We will assume that $G$ is such that $\theta$ lifts to an involution of $G$, and by abuse of notation we also denote this involution by $\theta$.  Then we let $K$ be any quasireductive subgroup of $G$ which satisfies $(G^\theta)^\circ\sub K\sub G^\theta$, where $(-)^\circ$ denotes the connected component of the identity.  

We recall that for a given choice of $G,\theta$, and $K$, we set $K'$ to be the quasireductive subgroup of $G$ with $K'_0=K_0$ and $\operatorname{Lie}K'=\k':=\k_{\ol{0}}\oplus\p_{\ol{1}}$.  Such a subgroup exists and is unique by the theory of super Harish-Chandra pairs (see \cite{masuoka2022group}).

With the above setup, we will consider the homogeneous supervarieties $G/K$ and $G/K'$, which will be smooth, affine supervarieties (see, for instance, \cite{masuoka2018geometric}).

\subsection{Convention on actions}  Since $G$ has a left action on $G/K$ as a space, $\Bbbk[G/K]$ naturally carries a right action as a $G$-module.  The action of $\UU\g$ on $\Bbbk[G/K]$ is given by 
\[
u\mapsto (u\otimes 1)\circ a^*,
\]
where $a:G\times G/K\to G/K$ is the action map.  This defines an algebra homomorphism $\UU\g\to\DD(G/K)^{op}$, where $\DD(G/K)$ is the algebra of differential operators on $G/K$. 

\begin{remark}\label{remark left right action}
 We work with this right action so that the action on distributions will be a left action.  However, note that if $V=L(\lambda)$ is a left-module of highest weight $\lambda$ for $\UU\g$, then as a right module (obtained via precomposition with the antipode), we have $V=L(-\lambda)$, i.e.~ $V$ becomes a module of highest weight $-\lambda$.  
\end{remark}
\subsection{Ghost distributions}

For an affine supervariety $X$ and point $x\in X(\Bbbk)$, we set
\[
\Dist(X,x):=\{\psi:\Bbbk[X]\to \Bbbk:\psi(\m_x^n)=0\text{ for }n\gg0\}.
\]
By our conventions above, $\Dist(G/K,eK)$ has the natural structure of a left $\g$-module under precomposition by vector fields.  We recall that its structure is given by
\[
\UU\g/\UU\g\k\cong\Dist(G/K,eK)
\]
from the map
\[
u\mapsto\operatorname{res}_{eK}\circ\left(u\otimes 1\right)\circ a^*,
\]
where $a:G\times G/K\to G/K$ is the action morphism.  Since $K_0$ stabilizes $eK$, the action of $\k_{\ol{0}}$ integrates to an action of $K_0$ on $\Dist(G/K,eK)$; thus we obtain an action of $K'$ on $\Dist(G/K,eK)$.  The central observation of \cite{sherman2021ghost} is the existence of a natural isomorphism of $K'$-modules
\[
\Dist(G/K,eK)\cong\Ind_{\k_{\ol{0}}}^{\k'}\Dist(G_0/K_0,eK_0).
\]
This isomorphism is induced by the natural inclusion $\Dist(G_0/K_0,eK_0)\sub\Dist(G/K,eK)$.

For the pairs we consider, we have that $\Lambda^{\dim\p_{\ol{1}}}\p_{\ol{1}}$ is a trivial $\k_{\ol{0}}$-module, and thus we have an isomorphism of even vector spaces
\[
\Dist(G_0/K_0,eK_0)^{K_0}\to \Dist(G/K,eK)^{K'}
\]
given explicitly by
\[
z\mapsto v_{\k'}\cdot z=zv_{\k'}.
\]
Here $v_{\k'}$ is a nonzero element of $(\UU\k'/\UU\k'\k_{\ol{0}})^{\k'}$, which is a one-dimensional vector space (see \cite{sherman2021ghost}).

\begin{definition}
	We define $\AA_{G/K}:=\Dist(G/K,eK)^{K'}$ to be the ghost distributions on $G/K$.
\end{definition}

\subsection{The algebraic approach} Notice that one can work purely algebraically and consider $\UU\g/\UU\g\k$ as a $\k'$-module; in this case we set 
\[
\AA_{(\g,\k)}:=(\UU\g/\UU\g\k)^{\k'};
\]
if $K$ is connected then this agrees with $\AA_{G/K}$.  In general $\AA_{G/K}$ is a subspace of $\AA_{(\g,\k)}$ given by the invariants under $K_0/K_0^{\circ}$.  Here, our above isomorphism of $\Dist(G_0/K_0,eK_0)^{K_0}$ with $\Dist(G/K,eK)^{K'}$ becomes an isomorphism
\[
\left(\UU\g_{\ol{0}}/\UU\g_{\ol{0}}\k_{\ol{0}}\right)^{\k_{\ol{0}}}\xto{\sim}\left(\UU\g/\UU\g\k\right)^{\k'},
\]
given by
\[
u\mapsto v_{\k'}u.
\]

\begin{remark}
	For the pairs we consider, $\AA_{G/K}$ and $\ZZ_{G/K}$ are always purely even vector spaces because $\dim\p_{\ol{1}}$ is even (since it admits a nondegenerate symplectic form).
\end{remark}

\subsection{Cartan subspaces and the Iwasawa decomposition}  We let $\a\sub\p_{\ol{0}}$ denote a Cartan subspace, that is a maximal subspace of $\p_{\ol{0}}$ consisting only of semisimple elements.  We may extend $\a$ to a $\theta$-stable Cartan subalgebra of $\g$, which we call $\h$, and then we have $\h=\mathfrak{t}\oplus\a$ according to the $(\pm1)$-eigenspaces of $\theta$ on $\h$.  Notice that $\h=\h_{\ol{0}}$ because of the choice of Lie superalgebras that we work with.  

Write $\Delta\sub\h^*$ for the roots of $\g$ with respect to $\h$, and let $\ol{\Delta}\sub\a^*\setminus\{0\}$ denote the collection of nonzero restrictions of roots to $\a$. Choose a decomposition $\ol{\Delta}=\ol{\Delta}^+\sqcup\ol{\Delta}^-$ into positive and negative roots, and set
\[
\n=\bigoplus\limits_{\alpha\in\ol{\Delta}^+}\g_{\alpha}.
\]
We say that $\g$ admits an Iwasawa decomposition if for some choice of positive roots in $\ol{\Delta}$ as above, we have $\g=\k\oplus\a\oplus\n$.  In this case, we choose a decomposition $\Delta=\Delta^+\sqcup\Delta^-$ of positive and negative roots in such a way that the restriction map from $\h^*$ to $\a^*$ sends $\Delta^+$ to $\ol{\Delta}^+\sqcup\{0\}$.  If $\b$ denotes the corresponding Borel subalgebra, then we say that $\b$ is an Iwasawa Borel subalgebra; observe that $\a\oplus\n\sub\b$ and thus $\b+\k=\g$.  

By \cite{sherman2020iwasawa}, at least one of $(\g,\k)$ or $(\g,\k')$ admits an Iwasawa decomposition.

\subsection{Highest weight functions}  We now assume that $(\g,\k)$ admits an Iwasawa decomposition with Iwasawa Borel subalgebra $\b$.  Let $\Lambda\sub\a^*$ denote the set of weights of rational $\b$-eigenfunctions on $G/K$.  Then from \cite{sherman2021ghost} we have the following facts:
\begin{itemize}
	\item $\Lambda$ is a full rank lattice in $\a^*$;
	\item if we write $\Bbbk(G/K)^{(\b)}$ for the subalgebra of rational $\b$-functions on $G/K$, then the restriction map 
	\[
	\Bbbk(G/K)^{(\b)}\to \Bbbk(G_0/K_0)^{(\b_{\ol{0}})}
	\]
	is an isomorphism.  In particular for each $\lambda\in\Lambda$ there exists a one-dimensional subspace of rational $\b$-eigenfunctions of weight $\lambda$;
	\item all nonzero rational $B$-eigenfunctions $f$ are regular in a neighborhood of $eK$ and satisfy $f(eK)\neq0$.
\end{itemize}

\begin{definition}
	For $\lambda\in\Lambda$, we denote by $f_{\lambda}$ the unique rational $\b$-eigenfunction on $G/K$ with $f_{\lambda}(eK)=1$.  
\end{definition}

\subsection{Highest weight submodules $V(\lambda)$} We let $\Lambda^+\sub\Lambda$ denote the set of $\lambda$ in $\Lambda$ such that $f_{\lambda}\in \Bbbk[G/K]$. By Section 5 of \cite{sherman2021spherical}, $\Lambda^+$ is Zariski dense in $\a^*$.

\begin{lemma}\label{lemma_submodule_gen}
	For $\lambda\in\Lambda^+$, set $V(\lambda):=\UU\g f_{\lambda}$.  Then we have 
	\[
	V(\lambda)=\UU\k f_{\lambda}=\UU\k' f_{\lambda}.
	\]
\end{lemma}

\begin{proof}
The first equality follows from the Iwasawa decomposition.  For the second, we note that we have $\g=\n+\c(\a)+\k'$; thus it suffices to show that $f_{\lambda}$ is a $\c(\a)$-eigenvector.  By \cite{sherman2020iwasawa}, $\c(\a)$ is generated by simple roots of $\b$ which are fixed by $\theta$, and thus it suffices to show that $\alpha$ is a negative simple root of $\b$ with $\theta\alpha=\alpha$, then $e_{\alpha}f_{\lambda}=0$, where $e_{\alpha}\in\g_{\alpha}$.  

If $\alpha$ is even or odd nonisotropic, then this follows because $(\alpha,\lambda)=0$ and the representation theory of $\s\l(2)$ and $\o\s\p(1|2)$.  If $\alpha$ is odd isotropic, then because $(\alpha,\lambda)=0$, $e_{\alpha}f_{\lambda}$ will be an odd $\b$-highest weight vector of $\Bbbk[G/K]$; however by Section 5 of \cite{sherman2021spherical}, this implies $e_{\alpha}f_{\lambda}=0$, and we are done.
\end{proof}

For the rest of the paper we will write $V(\lambda)$ in place of $\UU\g f_{\lambda}$ when $\lambda\in\Lambda^+$.

\subsection{Harish-Chandra morphism}

Under our assumption of an Iwasawa decomposition, we obtain a decomposition of $\UU\g$ as $\UU\g=S(\a)\oplus(\n\UU\g+\UU\g\k)$.   Thus we have
\[
\Dist(G/K,eK)\cong S(\a)\oplus (\n\UU\g+\UU\g\k)/\UU\g\k.
\] 
We define the Harish-Chandra morphism to be the projection of $\Dist(G/K,eK)$ onto $S(\a)$ along $(\n\UU\g+\UU\g\k)/\UU\g\k$, postcomposed with the antipode $\sigma_{\a}$ of $\a$.  The purpose of this postcomposition with the antipode is to make our formulas look familiar to those who work with left $\UU\g$ modules as opposed to right $\UU\g$-modules (see \cref{remark left right action}). The following lemma is straightforward.

\begin{lemma}
	For $\psi\in\Dist(G/K,eK)$ we have
	\[
	\psi(f_{\lambda})=HC(\psi)(-\lambda).
	\]
\end{lemma}

\begin{cor}\label{cor_nec_condition_extns}
	Suppose that $\mu,\lambda\in\Lambda^+$ and $f_{\mu}\in V(\lambda)$.  Then there exists $c\in \Bbbk$ such that $HC(\gamma)(-\mu)=cHC(\gamma)(-\lambda)$ for all $\gamma\in\AA_{G/K}$.
\end{cor}

\begin{proof}
	By Lemma \ref{lemma_submodule_gen}, we may write $f_{\mu}=uf_{\lambda}$ for some $u\in\UU\k'$.  Thus for $\gamma\in\AA_{G/K}$ we have
	\[
	HC(\gamma)(-\mu)=\gamma(f_{\mu})=\gamma(uf_{\lambda})=\varepsilon(u)\gamma(f_{\lambda})=\varepsilon(u)HC(\gamma)(-\lambda).
	\]
	Here $\varepsilon:\UU\k'\to \Bbbk$ is the counit.
\end{proof}

\begin{remark}
	Corollary \ref{cor_nec_condition_extns} implies that in order for $f_{\mu}\in V(\lambda)$, the evaluations $\operatorname{ev}_{-\lambda},\operatorname{ev}_{-\mu}$ must be linearly dependent on $HC(\AA_{G/K})\sub S(\a)$.  For $\lambda\in\h^*$, we write $\operatorname{ev}_{\lambda}:S\h\to\C$ for the linear map $\operatorname{ev}_{\lambda}(f)=f(\lambda)$.
\end{remark}

\subsection{Branching to $K'$}

The following lemmas are consequences of \cite{sherman2021ghost}.

\begin{lemma}\label{lemma branching}
	For $\lambda\in\Lambda^+$, $\operatorname{Res}_{K'} V(\lambda)$ contains at most one copy of $I_{K'}(\Bbbk)$; further it contains a copy if and only if $HC(\gamma)(\lambda)\neq0$ for some $\gamma\in\AA_{G/K}$.
\end{lemma}

\begin{lemma}\label{lemma_proj_nonzero}
	For $\lambda\in\Lambda^+$, if $V(\lambda)$ is irreducible and contains $I_{K'}(\Bbbk)$, then $I_{G}(L(\lambda))\sub \Bbbk[G/K']$.
\end{lemma}

\section{Injectivity of the Harish-Chandra Homomorphism}

This section is devoted to the proof and consequences of the following result:
\begin{thm}\label{thm_injectivity}
Suppose that $(\g,\k)$ admits an Iwasawa decomposition.  Then the map
	\[
	HC:\AA_{(\g,\k)}\to S(\a)
	\]
	is injective.
\end{thm}

\begin{proof}
We consider the filtration on $\Dist(G/K,eK)$ defined in Sec.~ 8.6 of \cite{sherman2021ghost}, except that we `halve' the indexing; i.e.~ our filtration on $\Dist(G/K,eK)$ will be indexed by half-integers, and if an element previously lied in the $r$th part of the filtration, then it now lies in the $r/2$ part of the filtration by definition.  Then Lem. 8.18 of \cite{sherman2021ghost} tells us that if $f\in\Dist(G/K,eK)$ lies in the $r$th part of the filtration, we must have $\deg HC(f)\leq r$.

Now we may represent any element of $\AA_{(\g,\k)}$ as $v_{\k'}z\in\UU\g/\UU\g\k$, where 
$z\in(\UU\g_{\ol{0}}/\UU\g_{\ol{0}}\k_{\ol{0}})^{\k_{\ol{0}}}$.  Let $\Delta_{\ol{1}}^+$ denote the collection of odd positive roots, and let $S\sub\Delta_{\ol{1}}^+$ denote the subset of those roots $\alpha$ for which $\theta\alpha\neq\alpha$.  Then $-\theta$ is an involution on $S$, and by Section 6 of \cite{sherman2020iwasawa}, it admits no fixed points.  Thus let $\alpha_1,\dots,\alpha_k\in S$ be distinct representatives of the orbits of $-\theta$ on $S$; then $\p_{\ol{1}}$ admits a basis given by
\[
x_i=e_{\alpha_i}-\theta e_{\alpha_i}, \ \ \ \ y_i=e_{-\alpha_i}-\theta e_{-\alpha_i}.
\]
Therefore we may represent $v_{\k'}$ in the following way
\[
v_{\k'}=x_ky_k\cdots x_1y_1+l.o.t.
\]
where $l.o.t.$ are lower order terms in our filtration, given by monomials in our basis above.  Thus we may write
\[
v_{\k'}z=x_ky_k\cdots x_1y_1z+l.o.t.
\]
Suppose that $z$ is of degree $r$; then it suffices to show that 
\[
\deg HC(x_ky_k\cdots x_1y_1z)=k+r.
\]
In other words we can work up to terms of degree lower than $k+r$.  Begin by writing $x_k=(e_{\alpha_k}+\theta e_{\alpha_k})-2e_{\alpha_k}$; then since $e_{\alpha_k}\in\n$, we obtain
\[
HC(x_ky_k\cdots x_1y_1z)=HC((e_{\alpha_k}+\theta e_{\alpha_k})y_k\cdots x_1y_1z).
\]
Set $h_{\ol{\alpha_i}}=h_{\alpha_i}-\theta h_{\alpha_i}$; then we have
\[
[e_{\alpha_k}+\theta e_{\alpha_k},y_k]=h_{\ol{\alpha_k}}+(1-\theta)[\theta e_{\alpha},e_{-\alpha}].
\]
Thus we have
\begin{eqnarray*}
HC(x_ky_k\cdots x_1y_1z)& = &HC(h_{\ol{\alpha_k}}x_{k-1}y_{k-1}\cdots x_1y_1z)+HC((1-\theta)[\theta e_{\alpha},e_{-\alpha}]x_{k-1}y_{k-1}\cdots x_1y_1z)\\
                        & + &\sum(\pm)HC(y_k\cdots [e_{\alpha_k}+\theta e_{\alpha_k},w_i]y_i\cdots z)\\
                        & + &HC(y_k\cdots x_1y_1[e_{\alpha_k}+\theta e_{\alpha_k},z]).
\end{eqnarray*}
In the above sum $w_i$ is either $x_i$ or $y_i$.  We will show that all but the first term will have degree strictly less than $r+k$.  Starting with the second term, we may write $(1-\theta)[\theta e_{\alpha},e_{-\alpha}]=n+k$, where $n\in\n$ and $k\in\k$.  Then we have
\begin{eqnarray*}
HC((1-\theta)[\theta e_{\alpha},e_{-\alpha}]x_{k-1}y_{k-1}\cdots x_1y_1z)& = &HC(x_{k-1}y_{k-1}\cdots x_1y_1[k,z])\\
                                                                         & + &\sum HC(x_{k-1}\cdots [k,w_i]\cdots y_1z).
\end{eqnarray*}
Since $k$ is even, $[k,w_i]$ will continue to be odd, and thus $HC(x_{k-1}\cdots [k,w_i]\cdots y_1z)$ will be of degree at most $r+k-1$.  Further, $[k,z]$ will be of filtered degree $r-1$, and thus $HC(x_{k-1}y_{k-1}\cdots x_1y_1[k,z])$ will be of degree at most $r+k-1$ as well.  This deals with the second term in our large sum.

For the third term, we observe that if $\alpha_i\neq\alpha_k$, then 
\[
[e_{\alpha_k}+\theta e_{\alpha_k},w_i]=(1-\theta)f_{\beta}+(1-\theta)f_{\gamma},
\]
where either $\beta=\alpha_k\pm\alpha_i,\gamma=\alpha_k\pm\theta\alpha_i$ are roots and $f_{\beta},f_{\gamma}$ are root vectors, or $\beta,\gamma$ are not roots and $f_{\beta},f_{\gamma}$ are zero.  In any case $\beta,\gamma\neq0$, and $f_{\beta},f_{\gamma}$ will be even root vectors.  Thus we may write $[e_{\alpha_k}+\theta e_{\alpha_k},w_i]=n+k$ where $n\in\n_{\ol{0}}$ and $k\in\k_{\ol{0}}$.  Now we move $n$ all the way to the left and $k$ all the way to the right, and because they are even the terms we obtain will all have filtered degree at most $r+k-1$, as desired.

Finally, we observe that $[e_{\alpha_k}+\theta e_{\alpha_k},z]$ is odd and will live in the $r-\frac{1}{2}$ part of our filtration.  Thus the third and last term in our sum will have degree at most $r+k-1$ once again; thus we have shown that
\[
HC(x_ky_k\cdots x_1y_1z)=HC(h_{\ol{\alpha_k}}x_{k-1}y_{k-1}\cdots x_1y_1z)+l.o.t.
\]
Because $\h$ normalizes $\n$, we may continue inductively to obtain that
\[
HC(x_ky_k\cdots x_1y_1z)=HC(h_{\ol{\alpha_k}}\cdots h_{\ol{\alpha_1}}z)+l.o.t.
\]
To finish, we write $z=n+HC(z)+k$, where $n\in\n\UU\g$ and $k\in\UU\g\k$, and since $\h$ normalizes $\n$ we find that
\[
HC(x_ky_k\cdots x_1y_1z)=(-1)^{k}h_{\ol{\alpha_k}}\cdots h_{\ol{\alpha_1}}HC(z)+l.o.t.
\]
as desired.
\end{proof}

\begin{remark}\label{rmk injectivity alldridge}
	Note that one of the consequences of \cite{alldridge2012harish} is that $HC$ is injective on $\ZZ_{G/K}$.
\end{remark}

The following is a consequence of the proof of Theorem \ref{thm_injectivity}:
\begin{cor}\label{cor highest order term}
  With $\alpha_1,\dots,\alpha_k$ as in the proof of Theorem \ref{thm_injectivity}, we have
	\[
	HC(v_{\k'}z)=(-1)^{k}h_{\ol{\alpha_k}}\cdots h_{\ol{\alpha_1}}HC(z)+l.o.t.
	\]
\end{cor}

\begin{definition}
	We say a weight $\lambda\in\Lambda^+$ is $(\g,\k)$-typical if $\operatorname{Res}_{K'}V(\lambda)$ contains a copy of $I_{K'}(\Bbbk)$.
\end{definition}
By the work of \cite{sherman2021ghost}, $(\g\times\g,\g)$-typical weights are in natural bijection with typical dominant integrable weights of $\g$, which explains the terminology.

\begin{cor}
	The set of $(\g,\k)$-typical weights in $\Lambda^+$ is given by the intersection of $\Lambda^+$ with a nonempty, Zariski open subset of $\a^*$.  
\end{cor}
\begin{proof}
 Let $U\sub\a^*$ be the union of the nonvanishing sets of $HC(D)$ for $D\in\AA_{G/K}$.  Then $U\cap \Lambda^+$ will consist exactly of the $(\g,\k)$-typical weights.
\end{proof}

\begin{cor}\label{cor generic typicality}
	There exists a Zariski open subset $U$ of $\a^*$ such that for all $\lambda\in U\cap\Lambda^+$:
	\begin{enumerate}
		\item[(i)] $V(\lambda)$ is irreducible;
		\item[(ii)] $\lambda$ is $(\g,\k)$-typical.
	\end{enumerate}
\end{cor}
\begin{proof}
	Only the generic irreducibility needs to be justified; however this follows from Thm. 6.4.4, \cite{sherman2020sphericalthesis}.
\end{proof}

\section{List of Supersymmetric Pairs and Interlacing Automorphisms}

In this section we will be working purely algebraically, without reference to any choice of a specific global supersymmetric space $G/K$.  

\subsection{Table of supersymmetric pairs}  We begin this section with the following table of all supersymmetric pairs that are of the type that we consider.  The table states whether they satisfy the Iwasawa decomposition, and describes the GRS (generalized root system) automorphism that $\theta$ induces on $\h^*=\mathfrak{t}^*\oplus\a^*$.  The notation we use below for root systems (along with the table) is taken from \cite{sherman2020iwasawa}.

\renewcommand{\arraystretch}{2}
\begin{center}
\begin{tabular}{|c|c|c|}
	\hline 
	Supersymmetric Pair & Iwasawa Decomposition? & GRS Automorphism \\
	\hline
	\makecell{$(\g\l(m|n)$,\\ $\g\l(r|s)\times\g\l(m-r|n-s))$} & Iff $(m-2r)(n-2s)\geq0$   & \makecell{$\epsilon_i\leftrightarrow\epsilon_{m-i+1}, 1\leq i\leq r$,\\ $\delta_j\leftrightarrow\delta_{n-j+1}, 1\leq j\leq s$} \\
	\hline
	$(\g\l(m|2n),\o\s\p(m|2n))$ & Yes &  $\epsilon_i\leftrightarrow-\epsilon_i, \ \ \delta_i\leftrightarrow-\delta_{2n-i+1}$ \\
	\hline
	\makecell{$(\o\s\p(m|2n)$,\\$\o\s\p(r|2s)\times\o\s\p(m-r,2n-2s))$}& Iff $(m-2r)(n-2s)\geq0$  & \makecell{$\epsilon_i\leftrightarrow-\epsilon_i, 1\leq i\leq r$\\ $\delta_i\leftrightarrow\delta_{n-i+1}, 1\leq i\leq s$}\\
	\hline
	$(\o\s\p(2m|2n),\g\l(m|n))$  & Yes & $\delta_i\leftrightarrow-\delta_i,\ \epsilon_i\leftrightarrow-\epsilon_{m-i+1}$ \\
	\hline
	$(\mathfrak{d}(1,2;\alpha),\o\s\p(2|2)\times\s\o(2))$ &  Yes & $\epsilon\leftrightarrow-\epsilon,\delta\leftrightarrow-\delta$  \\
	\hline
	$(\mathfrak{ab}(1|3),\g\o\s\p(2|4))$ & Yes & $\epsilon_1\leftrightarrow	-\epsilon_1,\delta\leftrightarrow-\delta$\\
	\hline
	$(\mathfrak{ab}(1|3),\s\l(1|4))$ & Yes & $\epsilon_1\leftrightarrow-\epsilon_1,\epsilon_2\leftrightarrow-\epsilon_2,\delta\leftrightarrow-\delta$\\
	\hline
	$(\mathfrak{ab}(1|3),\mathfrak{d}(1,2;2))$ & Yes & $\epsilon_i\leftrightarrow-\epsilon_i$\text{ for all }$i$\\
	\hline
	$(\mathfrak{g}(1|2),\mathfrak{d}(1,2;3))$ & Yes &  $\epsilon_i\leftrightarrow-\epsilon_i$\text{ for all }$i$\\
	\hline
	$(\mathfrak{g}(1|2),\o\s\p(3|2)\times\s\l_2)$ & No &  $\epsilon_i\leftrightarrow-\epsilon_i$\text{ for all }$i$\\
	\hline
\end{tabular}
\end{center}

\subsection{Interlacing automorphisms $(\g,\k)\cong(\g,\k')$}

In the case that we have an isomorphism of supersymmetric pairs $(\g,\k)\cong(\g,\k')$, meaning that there exists an automorphism of $\g$ taking $\k$ to $\k'$, we will show we can take this isomorphism to be of a special form.  For a Cartan subspace $\a$ of $\g$ we write $A$ for the connected torus in $\operatorname{Inn}(\g)$ that it corresponds to, where $\operatorname{Inn}(\g)$ denotes the inner automorphisms of $\g$. 

\begin{definition}
	We say that an automorphism $\phi$ of $\g$ interlaces $(\g,\k)$ and $(\g,\k')$ with respect to a Cartan subspace $\a$ if:
	\begin{enumerate}
		\item[(i)] $\phi\theta=\delta\theta\phi$;
		\item[(ii)] $\phi=\Ad(t)$ is inner with $t\in A$, we have $\Ad(t^2)=\delta$, $t^4=1$.  In particular, $\theta t^2=t^2$.
	\end{enumerate}
In this case we call $\phi$ an interlacing automorphism, and say that $(\g,\k)$ is interlaced.
\end{definition}
In particular an interlacing automorphism satisfies $\phi(\k)=\k'$ and $\phi(\k')=\k$.  It is easy to check that the inverse of an interlacing automorphism is again an interlacing automorphism.  Further, because an interlacing automorphism $\phi$ fixes $\a$ pointwise we have $\phi(\n)=\n$, and thus $\phi$ takes the Iwasawa decomposition $\g=\k\oplus\a\oplus\n$ to another Iwasawa decomposition $\g=\k'\oplus\a\oplus\n$.  

We have the following satisfying theorem.
\begin{thm}\label{thm_interlacing_pairs}
	Let $(\g,\k)$ be a supersymmetric pair from our above list; then the following are equivalent:
	\begin{enumerate}
		\item[(i)] $(\g,\k)$ admits an interlacing automorphism;
		\item[(ii)] $(\g,\k)$ is conjugate to $(\g,\k')$;
		\item[(iii)] both $(\g,\k)$ and $(\g,\k')$ admit Iwasawa decompositions.
	\end{enumerate}   
\end{thm}
The proofs of (i)$\implies$(ii) and (ii)$\implies$ (iii) are clear.  The proof of (iii)$\implies$(1) follows from studying the possible reduced generalized root systems that are obtained in the cases we need to consider.  We begin by recalling that the pair $(\Delta,\h^*)$ may be viewed as an irreducible generalized root system (GRS) in the sense of \cite{serganova1996generalizations}.   We may write $\Delta_{re}\sub\Delta$ for the roots that are nonisotropic, so that $\Delta_{re}\sub\h^*$ defines a root system in the classical sense.  In particular it decomposes into irreducible components $\Delta=\Delta_1\sqcup\cdots\sqcup\Delta_k$, and correspondingly $\h^*=V_0\oplus V_1\oplus\cdots\oplus V_k$, where $\Delta_i\sub V_i$ is an irreducible root system, and $V_0=(V_1\oplus\cdots\oplus V_k)^\perp$. 

Now the involution $\theta$ induces an automorphism on $\h^*$; if we have $\theta(V_i)=V_i$ then we obtain a projection map $p_i:\a^*\to\a^*\cap V_i$.  

\begin{lemma}
	We have $0\notin\Delta_{\ol{1}}|_{\a}$ if and only if for some $i$ we have $\theta V_i=V_i$ and $0\notin p_i(\Delta_{\ol{1}})$.
\end{lemma}

\begin{proof}
	The backwards direction is clear; for the forward direction we work case by case.  Note that by the proof of Prop. 3.1 of \cite{sherman2020iwasawa}, all pairs have $\theta V_i\sub V_i$ for all $i$ except possibly when $\g=\o\s\p(4|2n)$. 
\end{proof}

\begin{prop}
	If $0\notin\Delta_{\ol{1}}|_{\a^*}$, then there exists an $i$ such that the following hold: 
	\begin{enumerate}
		\item[(i)] $\theta V_i=V_i$;
		\item[(ii)] if $A_i\sub A$ denotes the subgroup determined by $V_i\cap \a^*$, then there exists $t\in A_i$ such that $\phi=\Ad(t)$ is an interlacing automorphism.
	\end{enumerate}
\end{prop}

\begin{proof}
	We will show that in each case we can find $t\in A$ such that $\Ad(t)$ acts by $\pm i$ on the root spaces of $\g_{\ol{1}}$, and by $\pm1$ on root spaces of $\g_{\ol{0}}$. In particular we will have $\Ad(t^{-1})=\delta\circ\Ad(t)$.  Since $\theta\circ\Ad(t)\circ\theta=\Ad(t^{-1})$, this will give $\theta\Ad(t)=\delta\Ad(t)\theta$, and further it is obvious that $\Ad(t)$ fixes $\a$.  Finally, $\Ad(t^2)$ will act by the identity on $\g_{\ol{0}}$ and have eigenvalue $(-1)$ on $\g_{\ol{1}}$, implying that it is $\delta$.    
	
	The proof is done in cases.  First of all we always have $\theta V_0=V_0$ when $V_0\neq0$, and if $0\notin p_0(\Delta_{\ol{1}})$ then as is shown in \cite{serganova1996generalizations}, $A_0$ is a one-dimensional torus and $\Delta_{\ol{1}}=\Delta_{\ol{1}}'\sqcup\Delta_{\ol{1}}''$ consists of two components on which $A_0$ acts faithfully by dual characters.  Thus we may let $t\in A_0$ be an order 4 element which acts by $i$ on $\Delta_{\ol{1}}'$ and $(-i)$ on $\Delta_{\ol{1}}''$.  These cases are given by $(\g\l(m|2n),\o\s\p(m|2n))$ and $(\o\s\p(2|2n),\o\s\p(1|2s)\times\o\s\p(1|2n-2s))$.  (Such pairs are called \emph{special}; they will be studied further in Section 7.)
	
	For the remaining cases we need to take an $i>0$.  Write $\ol{\Delta_{j}^i}$ for the image of $\Delta_{j}$ in $\a^*\cap V_i$. We will see that there always exists an $i>0$ such that the following happens:
	\begin{itemize}
		\item $\theta V_i=V_i$;
		\item the image of $\ol{\Delta_{\ol{1}}^i}$ is a single Weyl group orbit up to sign; and
		\item the lattice generated by $\ol{\Delta_{\ol{0}}^i}$ is of index two inside of the lattice generated by $\ol{\Delta_{\ol{1}}^i}$.
	\end{itemize}
	Given the above properties, we may choose $t\in A_i$ which acts by $\pm i$ on $\ol{\Delta_{\ol{1}}^i}$ and by $\pm1$ on $\ol{\Delta_{\ol{0}}^i}$.  
	
	Now we go through all the remaining supersymmetric pairs satisfying $0\notin\Delta_{\ol{1}}|_{\a}$.  Recall that we follow the notation of \cite{sherman2020iwasawa}.
	
	\begin{itemize}
		\item $(\g\l(m|2n),\g\l(m-r|n)\times\g\l(r|n))$: the component $V_2$ of the root system spanned by $\delta_1,\dots,\delta_{2n}$ is $\theta$-stable and intersects $\a^*$ with basis $\ol{\delta_1},\dots,\ol{\delta_n}$, where $\ol{\delta_i}:=\frac{1}{2}(\delta_i-\delta_{n+i})$.  Then $\ol{\Delta_{\ol{1}}^2}$ consists of $\pm\delta_i$, for all $i$, while $\ol{\Delta_{\ol{0}}^2}$ consists of $\pm\delta_i\pm\delta_j$.
		\item $(\o\s\p(2m|2n),\g\l(m|n))$: the component $V_2$ of the root system spanned by $\delta_1,\dots,\delta_n$ is $\theta$-stable and entirely contained in $\a^*$.  Here $\ol{\Delta_{\ol{1}}^1}$ consists of $\pm\delta_1,\dots,\pm\delta_n$, while $\ol{\Delta_{\ol{0}}^2}$ consists of $\pm\delta_i\pm\delta_j$.
		\item $(\o\s\p(2m|2n),\o\s\p(m|2s)\times\o\s\p(m|2n-2s))$: the component $V_1$ spanned by $\epsilon_1,\dots,\epsilon_m$ is $\theta$-stable and entirely contained in $\a^*$.  Here $\ol{\Delta_{\ol{1}}^1}$ consists of $\pm\epsilon_1,\dots,\pm\epsilon_m$, while $\ol{\Delta_{\ol{0}}^1}$ consists of $\pm\epsilon_i\pm\epsilon_j$ for $i\neq j$.
		\item $(\o\s\p(m|4n),\o\s\p(r|2n)\times\o\s\p(m-r|2n))$: the component $V_2$ spanned by $\delta_1,\dots,\delta_{2n}$ is $\theta$-stable and intersects $\a^*$ with basis $\ol{\delta_1},\dots,\ol{\delta_n}$, where $\ol{\delta_i}:=\frac{1}{2}(\delta_i-\delta_{n+i})$.  Then $\ol{\Delta_{\ol{1}}^2}$ consists of $\pm\ol{\delta_1},\dots,\pm\ol{\delta_{n}}$, while $\ol{\Delta_{\ol{0}}^2}$ consists of $\pm\delta_i\pm\delta_j$.
		\item $(\mathfrak{d}(1,2;\alpha),\o\s\p(2|2)\times\s\o(2))$: for this case the component $V_1$ spanned by $\epsilon$ is $\theta$-stable and entirely contained in $\a^*$.  We have $\ol{\Delta_{\ol{1}}^1}=\{\pm\epsilon\}$ while $\ol{\Delta_{\ol{0}}^1}=\{\pm2\epsilon\}$.
		\item $(\a\b(1|3),\g\o\s\p(2|4))$ and $(\a\b(1|3),\s\l(1|4))$: for these two cases, the component $V_2$ spanned by $\delta$ is $\theta$-stable and entirely contained within $\a^*$.  We have $\ol{\Delta_{\ol{1}}^2}=\{\pm\frac{1}{2}\delta\}$ while $\ol{\Delta_{\ol{0}}^2}=\{\pm\delta\}$.
		\item $(\a\b(1|3),\mathfrak{d}(1,2;2))$: here the component $V_1$ spanned by $\epsilon_1,\epsilon_2,\epsilon_3$ is $\theta$-stable and entirely contained in $\a^*$.  We have $\ol{\Delta_{\ol{1}}^1}$ consists of $\frac{1}{2}(\pm\epsilon_1\pm\epsilon_2\pm\epsilon_3)$, and $\ol{\Delta_{\ol{0}}^1}$ consists of $\pm\epsilon_i\pm\epsilon_j$ for $i\neq j$.
	\end{itemize}
\end{proof}

\section{The Ghost Centre on $G/K$}

\subsection{General product structures} We now explain product structures relating the spaces $\ZZ_{G/K},\ZZ_{G/K'},\AA_{G/K},$ and $\AA_{G/K'}$.  As noted in \cite{sherman2021ghost}, we have natural, well-defined multiplication maps
\[
\AA_{G/K}\otimes\ZZ_{G/K}\to\AA_{G/K}, \ \ \ \ \ZZ_{G/K'}\otimes\AA_{G/K}\to\AA_{G/K}.
\]
Further, multiplication induces well-defined maps
\[
\AA_{G/K}\otimes\AA_{G/K'}\to \ZZ_{G/K'}, \ \ \ \ \AA_{G/K'}\otimes\AA_{G/K}\to\ZZ_{G/K}.
\]
If we further suppose that $(\g,\k)$ admits an Iwasawa decomposition $\g=\k\oplus\a\oplus\n$, then we have a commutative diagram
\[
\xymatrix{\AA_{G/K}\otimes\ZZ_{G/K}\ar[rr]\ar[d]^{HC\otimes HC} & & \AA_{G/K}\ar[d]^{HC} \\ S(\a)\otimes S(\a) \ar[rr]^{m} & & S(\a)}
\]
where $m$ denotes multiplication on $S(\a)$.

\subsection{Further structure for interlaced pairs}  Let us now assume that $(\g,\k)$ satisfies the hypotheses of Theorem \ref{thm_interlacing_pairs}, i.e.~ it is interlaced.  If we choose an Iwasawa decomposition $\g=\k\oplus\a\oplus\n$ for $\g$, we obtain a decomposition
\[
\UU\g=S(\a)+(\n\UU\g+\k\UU\g).
\]
Let $\phi$ be an interlacing automorphism from $(\g,\k)$ to $(\g,\k')$ with respect to $\a$.  First of all, $\phi$ induces isomorphisms on distributions:
\[
\phi:\UU\g/\UU\g\k\xto{\sim}\UU\g/\UU\g\k', \ \ \ \ \phi:\UU\g/\UU\g\k'\xto{\sim}\UU\g/\UU\g\k
\]
which twist the action of $\g$.    This isomorphism in turn induces isomorphisms
\[
\ZZ_{G/K}\xto{\sim}\ZZ_{G/K'}, \ \ \ \ \AA_{G/K}\xto{\sim}\AA_{G/K'}.
\]
Now $\phi^2=\delta$, and thus since $\ZZ_{G/K}$ and $\AA_{G/K}$ are purely even super vector spaces we have that $\phi^2$ restricts to the identity on $\ZZ_{G/K}$, and $\AA_{G/K}$; thus the compositions
\[
\ZZ_{G/K}\xto{\phi}\ZZ_{G/K'}\xto{\phi}\ZZ_{G/K}, \ \ \ \AA_{G/K}\xto{\phi}\AA_{G/K'}\xto{\phi}\AA_{G/K},
\]
are the identity; in particular we have equalities $\phi=\phi^{-1}$ on these spaces. 

As noted previously, $\phi$ takes $\k\oplus\a\oplus\n$ to $\k'\oplus\a\oplus\n$.  From this we obtain a commutative diagram
\[
\xymatrix{\UU\g/\UU\g\k \ar[dr]_{HC} \ar[r]^{\phi} & \UU\g/\UU\g\k'\ar[d]^{HC} \\ & S(\a)}
\]

\subsection{The ghost centre $\widetilde{\ZZ_{G/K}}$}
\begin{lemma}
	We have $\AA_{G/K}\cap\ZZ_{G/K}=0$.
\end{lemma}
\begin{proof}
	For all the pairs we consider, $[\k,\k']\cap\p_{\ol{0}}\neq0$.  Thus if $u\in\AA_{G/K}\cap\ZZ_{G/K}$ it would be invariant under some nonzero $p\in\p_{\ol{0}}$; however $p$ admits no invariants on $\Dist(G/K,eK)$.
\end{proof}

\begin{prop}\label{ghost algebra prop}
	If $(\g,\k)$ is interlacing, then $\widetilde{\ZZ_{G/K}}:=\ZZ_{G/K}\oplus\AA_{G/K}$ admits the natural structure of a commutative algebra such that $\ZZ_{G/K}$ is a subalgebra,
	\[
	\ZZ_{G/K}\AA_{G/K}=\AA_{G/K}\ZZ_{G/K}\sub\AA_{G/K}, \ \text{ and } \ \AA_{G/K}\AA_{G/K}\sub\ZZ_{G/K}.
	\]
	With this algebra structure, $HC:\widetilde{\ZZ_{G/K}}\to S(\a)$ becomes a homomorphism of algebras.  In particular, 
	\[
	HC(\AA_{G/K})HC(\AA_{G/K})\sub HC(\ZZ_{G/K}).
	\]
\end{prop}

\begin{proof}
	To simply notation, for an element in $u+\UU\g\k\in\UU\g/\UU\g\k$ we will simply write $u$; we will write $u\cdot v$ for the product in the algebra $\widetilde{\ZZ_{G/K}}$, and $uv$ for the product of $u$ and $v$ thought of as elements in $\UU\g$ (assuming it is well-defined).  We will also write $\phi$ for an interlacing automorphism of $(\g,\k)$.
	
	We define the algebra structure as follows: given $z\in\ZZ_{G/K}$,  $u_{1},u_2\in\AA_{G/K}$, we set:
	\[
	u_1\cdot z=u_1z, \ \ \ z\cdot u_1=\phi(z)u_1, \ \ \ u_1\cdot u_2=\phi(u_1)u_2.
	\]
	Checking associativity is straightforward, using that $\phi=\phi^{-1}$ on these spaces.  To prove commutativity, we use injectivity of $HC$ from Theorem \ref{thm_injectivity} and Remark \ref{rmk injectivity alldridge} restricted to $\ZZ_{G/K}$ and $\AA_{G/K}$ separately; indeed, $HC(\phi(u_1)u_2)=HC(\phi(u_2)u_1)$, and $HC(u_1z)=HC(\phi(z)u_1)$.
\end{proof}

\begin{definition}
	We define $\widetilde{\ZZ_{G/K}}$ to be the ghost center of $G/K$.  We write $\widetilde{\ZZ_{(\g,\k)}}=\ZZ_{(\g,\k)}+\AA_{(\g,\k)}$ for the ghost centre of $(\g,\k)$.
\end{definition}
Clearly $\widetilde{\ZZ_{G/K}}$ is the subalgebra of $\widetilde{\ZZ_{(\g,\k)}}$ given by the fixed points of $K_0$.

\begin{remark}
	Observe that although our definition of the product structure on $\widetilde{\ZZ_{G/K}}$ a priori depends on a choice of interlacing automorphism $\phi$, its definition is independent of $\phi$ due to the injectivity of $HC$ on $\ZZ_{G/K}$ and $\AA_{G/K}$.
\end{remark}

\section{Lifting $\widetilde{\ZZ_{G/K}}$ to Operators}

In this section we explain how to lift $\widetilde{\ZZ_{G/K}}$ to an algebra of operators on $\Bbbk[G/K]$.  In order to do this, we will need  to make the additional \textbf{assumption} that:
\[
(\g,\k) \ \text{ is interlaced with interlacing automorphism }\phi=\Ad(t) \text{ satisfying }t^2\in K(\Bbbk). \ \ (*)
\]  
\begin{remark}
	We note that the above condition is not so restrictive, and holds in most cases when the pair is interlacing.  We always have that $\theta t^2=t^2$, so if we set $K=G^\theta$ then it satisfies this condition.  However it is possible in some cases that $(G^\theta)^\circ$ does not contain $t^2$.
\end{remark}

The process to lift $z\in\ZZ_{G/K}$ to an operator on $G/K$ is well known, and we recall it now.  We write $a^*:\Bbbk[G]\to \Bbbk[G]\otimes \Bbbk[G]$ for the coproduct map on $G$, and note that $\Bbbk[G/K]=\Bbbk[G]^K$.  Then we lift $z$ to the operator $\tilde{z}$ via
\[
\tilde{z}=(1\otimes z)\circ a^*.
\]
The operator $\tilde{z}$ defines a $G$-equivariant map $\Bbbk[G/K]\to \Bbbk[G/K]$; it further defines a differential operator on $G/K$.  In this way we obtain an isomorphism of algebras
\[
\ZZ_{G/K}\to D^{G}(G/K),  \ \ \ \ z\mapsto\tilde{z},
\]
where $D^{G}(G/K)$ denotes the algebra of $G$-equivariant differential operators on $G/K$.  

\subsection{Lifting $\AA_{G/K}$ to operators}\label{sec lifting ops} Let $u\in\AA_{G/K}$; then in a similar fashion to above, we may consider the operator
\[
u':=(1\otimes u)\circ a^*.
\]
However, by the invariance properties of $u$, this defines an operator $u':\Bbbk[G/K]\to \Bbbk[G/K']$.  To `fix' this, we use an interlacing automorphism $\phi$.  Recall that $\phi$ is inner, with $\phi=\Ad(t)$ for $t\in A(\Bbbk)$; further $t$ satisfies that $t^2\in K(\Bbbk)$ by $(*)$.  If we write $R_t$ for the action on $G$ by right multiplication by $t$, then  it defines a $G$-equivariant isomorphism $R_t^*:\Bbbk[G/K']\to \Bbbk[G/K]$.  Further, since $t^2\in K(\Bbbk)$, we have that the composition
\[
\Bbbk[G/K]\xto{R_t^*}\Bbbk[G/K']\xto{R_t^*}\Bbbk[G/K]
\]
is the identity map.  Now we define 
\[
\tilde{u}:=R_t^*\circ u'=R_t^*\circ(1\otimes u)\circ a^*.
\]
Then $\tilde{u}$ defines a $G$-equivariant endomorphism of $\Bbbk[G/K]$.  

\begin{remark}[Caution]  The operator $\tilde{u}$ will \emph{not} be a differential operator on $G/K$; indeed, if it did then $\operatorname{res}_{eK}\circ\tilde{u}$ would live in $\Dist(G/K,eK)$.  However this element instead lives in $\Dist(G/K,aK)$.
\end{remark}

\begin{thm}\label{thm lifting to ops}
The maps $z\mapsto\tilde{z}$, $u\mapsto\tilde{u}$ define an injective morphism of algebras
\[
\widetilde{\ZZ_{G/K}}\to\End_{G}(\Bbbk[G/K]).
\]
Further, for $\lambda\in\Lambda^+$ we have $\frac{\tilde{u}(f_{\lambda})}{f_{\lambda}}=(e^\lambda(t))HC(u)(-\lambda)=\pm HC(u)(-\lambda)$.
\end{thm}

\begin{proof}
 	Let $z\in\ZZ_{G/K}$ and $u\in\AA_{G/K}$.  Then we have
 	\begin{eqnarray*}
 	\tilde{u}\tilde{z}& = &R_t^*\circ(1\otimes u)\circ a^*\circ(1\otimes z)\circ a^*\\
 	                  & = &R_t^*\circ(1\otimes (u\circ((1\otimes z)\circ a^*)))\circ a^*\\
 	                  & = &R_t^*\circ(1\otimes uz)\circ a^*\\
 	                  & = &\widetilde{uz};
 	\end{eqnarray*}
	\begin{eqnarray*}
	\tilde{z}\tilde{u}& = &(1\otimes z)\circ a^*\circ R_t^*\circ(1\otimes u)\circ a^*\\
	                  & = &R_t^*\circ(1\otimes(\phi(z)\circ((1\otimes u)\circ a^*)))\circ a^*\\
	                  & = &R_t^*\circ(1\otimes\phi(z)u)\circ a^*\\
	                  & = &\widetilde{\phi(z)u}.
	\end{eqnarray*}
Finally, if $u_1,u_2\in\AA_{G/K}$, then
	\begin{eqnarray*}
	\tilde{u_1}\tilde{u_2}& = &R_t^*\circ(1\otimes u_1)\circ a^*\circ R_t^*\circ(1\otimes u_2)\circ a^*\\
	                      & = &R_{t^2}^*\circ (1\otimes(\phi(u_1)\circ((1\otimes u_2)\circ a^*)))\circ a^*\\
	                      & = &(1\otimes\phi(u_1)u_2)\circ a^*\\
	                      & = &\widetilde{\phi(u_1)u_2}.
	\end{eqnarray*}
In the above we used that $t^2\in K(\Bbbk)$ so $R_{t^2}^*$ acts by the identity on $\Bbbk[G]^{K}$. 

To determine the action on $f_{\lambda}$ for $\lambda\in\Lambda^+$, we use that since $\tilde{u}$ is $G$-equivariant, $\tilde{u}f_{\lambda}$ must be a multiple of $f_{\lambda}$; to determine which multiple, we evaluate $\tilde{u}f_{\lambda}$ at $eK$.
\begin{eqnarray*}
	\operatorname{res}_{eK}\circ R_{t}^*\circ (1\otimes u)\circ a^*(f_{\lambda})& = &u\circ L_{t}^*\circ(\operatorname{res}_{eK}\otimes 1)\circ a^*(f_{\lambda})\\
	                                                             & = &u\circ L_t^*(f_{\lambda})\\
	                                                             & = &(e^{\lambda}(t))uf_{\lambda}\\
	                                                             & = &(e^{\lambda}(t))HC(u)(-\lambda) f_{\lambda}.                                                
\end{eqnarray*}
One can check that for all symmetric pairs we consider, $e^{\lambda}(t)=\pm1$ for all $\lambda\in\Lambda$.

Finally, to prove injectivity of our morphism $\widetilde{\ZZ_{G/K}}\to\End_G(\Bbbk[G/K])$, we first of all note that it is injective individually on $\ZZ_{G/K}$ and $\AA_{G/K}$ because $HC$ is injective on each space. Further we cannot have $\tilde{z}=\tilde{u}$ for nonzero $z\in\ZZ_{G/K}$, $u\in\AA_{G/K}$, because $\tilde{z}$ is a differential operator while $\tilde{u}$ is not.  This forces injectivity.
\end{proof}

\section{Special Supersymmetric Pairs}

\begin{definition}
	A supersymmetric pair $(\g,\k)$ is called special if our interlacing automorphism $\phi=\Ad(t)$ may be taken to act by the identity on $\g_{\ol{0}}$.
\end{definition}

The special supersymmetric pairs are those with properties that are close to the pair $(\g\times\g,\g)$.  We obtain only two families of special pairs, namely $(\g\l(m|2n),\o\s\p(m|2n))$ and $(\o\s\p(2|2n),\o\s\p(1|2s)\times\o\s\p(1|2n-2s))$.  Nevertheless these spaces are of interest to understand and their extra structure warrants a further study.

Set $\Aut(\g,\g_{\ol{0}})$ to be those automorphisms of $\g$ which fix $\g_{\ol{0}}$ pointwise.  For the cases $\g=\g\l(m|n),\o\s\p(2|2n)$ we have $\Aut(\g,\g_{\ol{0}})\cong \Bbbk^*$.  For each $c\in \Bbbk^*$, we write $\phi_c$ for the corresponding automorphism of $\Aut(\g,\g_{\ol{0}})$.  Explicitly, if $\g=\g_{-1}\oplus\g_0\oplus\g_1$, then $\phi_c$ corresponds to the automorphism acting by:
\[
\phi_c:=c^{-1}\id_{\g_{-1}}\oplus\id_{\g_{0}}\oplus c\id_{\g_1}.
\]
Now for each $c\in \Bbbk^*$, $(\g,\phi_c(\k))$ will be a supersymmetric pair with involution $\theta_c:=\phi_c\theta\phi_c^{-1}$.  We have the relation $\theta_c=\theta_d$ if and only if $d=-c$, and the same relationship for the subalgebras $\phi_c(\k)$.  However we always have $\phi_c(\k)_{\ol{0}}=\k_{\ol{0}}$.  Write $K_c$ for the subgroup of $G$ with $\operatorname{Lie}K_c=\phi_c(\k)$, and $(K_c)_0=K_0$.

\begin{prop}
	Suppose that $(\g,\k)$ is a special supersymmetric pair.  If $c\neq\pm1$, we have a natural $K_c$-equivariant isomorphism
	\[
	\Dist(G/K,eK)\to\Ind_{\k_{\ol{0}}}^{\phi_c(\k)}\Dist(G_0/K_0,eK_0),
	\]
	thereby inducing an isomorphism of vector spaces
	\[
	\Dist(G_0/K_0,eK_0)^{K_0}\xto{\sim}\Dist(G/K,eK)^{K_c}, \ \ \ z\mapsto v_{\phi_c(\k)}\cdot z.
	\]
\end{prop}

\begin{proof}
	By \cite{sherman2021ghost}, it suffices to show that $\phi_c(\k)_{\ol{1}}+\k_{\ol{1}}=\g_{\ol{1}}$ for $c\neq\pm1$. For the special pairs, $\theta$ interchanges $\g_{-1}$ and $\g_{1}$, thus we write $\g_{\ol{1}}=\g_{-1}\oplus\g_1$, and we have 
	\[
	\phi_c(\k)=\{(u,\theta_c(u)):u\in\g_{-1}\}=\{(u,c^2\theta(u)):u\in\g_{-1}\}.
	\]
	From this the result follows.
\end{proof}

The above result may be expressed algebraically as saying that we have an isomorphism of $\phi_c(\k)$-modules:
\[
\Ind_{\k_{\ol{0}}}^{\phi_c(\k)}\left(\UU\g_{\ol{0}}/\UU\g_{\ol{0}}\k_{\ol{0}}\right)\cong \UU\g/\UU\g\k
\]
\subsection{The full ghost algebra of $G/K$}  Let us assume for the rest of this section that $(\g,\k)$ is a special supersymmetric pair.
\begin{definition}
	We set $\AA_{G/K,c}:=\Dist(G/K,eK)^{K_c}$, and define the full ghost algebra of $G/K$ to be $\AA_{G/K}^{full}:=\sum\limits_{c\in \Bbbk^\times}\AA_{G/K,c}$.
\end{definition}
Observe that $\AA_{G/K,c}=\AA_{G/K,-c}$.  

We may write $\Delta=\Delta_{-1}\sqcup\Delta_0\sqcup\Delta_1$, where $\Delta_i$ are the roots coming from $\g_i$.  Then $\theta$ defines a bijection $\Delta_{-1}\to\Delta_{1}$.  It is possible in this case to choose for Iwasawa Borel one with positive system satisfying $\Delta_{\ol{1}}^{\pm}=\Delta_{\pm1}$, and we do this.
\begin{prop}
	For all $c\in \Bbbk^\times$, the map $HC:\AA_{G/K,c}\to S(\a)$ is injective.
\end{prop}
\begin{proof}
	For $c=\pm1$, this follows from \cite{alldridge2012harish}.  Now suppose that $c\neq\pm1$.  Then $-\theta$ defines an involution on $\Delta_{-1}$ without fixed points; write $\{\alpha_1,\dots,\alpha_k\}\sub\Delta_{-1}$ for representatives. Then a basis for $(\phi_c(\k_c))_{\ol{1}}$ is given by
	\[
	x_i:=e_{\alpha_i}+c^{2}\theta e_{\alpha_i}, \ \ \ \ y_i=e_{-\theta\alpha_i}+c^{2}\theta e_{-\theta\alpha_i}.
	\]
	From here the proof is almost verbatim to the one given in Theorem \ref{thm_injectivity}.  In particular the important observations are that we may write $x_i=(e_{\alpha_i}+\theta e_{\alpha})+(1-c^{2})\theta e_{\alpha}$, and $\theta e_{\alpha}\in\n^+$.  Further,
	\[
	[x_i,y_i]=(1+\theta c^2)[e_{\alpha},e_{-\theta\alpha}]+(c^2+\theta)h_{\alpha},
	\]
	where $h_{\alpha}:=[e_{\alpha},\theta e_{-\theta\alpha}]$.  In the end we will find that for $z\in(\UU\g/\UU\g\k_{\ol{0}})^{K_0}$,
	\[
	HC(zv_{\k_c'})=r\ol{h_{\alpha_1}}\cdots \ol{h_{\alpha_k}}HC(z)+l.o.t.
	\]
	where $r\in \Bbbk^\times$ and $\ol{h_{\alpha_i}}$ is the projection of $h_{\alpha_i}$ to $\a$.
\end{proof}

The following conjecture is based on phenomenon observed in the diagonal case $(\g\times\g,\g)$ in \cite{sherman2021ghost}.
\begin{conj}
	We have $HC(\AA_{G/K,c})=HC(\AA_{G/K})\sub HC(\ZZ_{G/K})$ for all $c\neq\pm1$.
\end{conj}

\begin{prop}
The space $\AA_{G/K}^{full}$ admits the structure of a commutative algebra such that $\AA_{G/K}$ is naturally a subalgebra.
\end{prop}
\begin{proof}
 Set $\AA_{G/K_c,d}:=\Dist(G/K_c,eK_c)^{K_d}$.  Note that 
\[
\AA_{G/K_c,\pm c}=\ZZ_{G/K_c}, \ \ \ \ \AA_{G/K_c,\pm ic}=\AA_{G/K_c}.
\]
Then we have natural product maps:
\[
\AA_{G/K_c,d}\otimes\AA_{G/K_r,c}\to\AA_{G/K_r,d},
\]
and isomorphisms
\[
\AA_{G/K_c,d}\xto{\phi_{c^{-1}}}\AA_{G/K,c^{-1}d}.
\]
Further we observe that $\phi_{-1}=\phi_1=\id$ on $\AA_{G/K_c,d}$ for all $c,d$; it follows that $\phi_{c}=\phi_{-c}$.   We define the product structure on $\AA_{G/K}^{full}$ as follows: for $u_c\in\AA_{G/K,c}$, $u_d\in\AA_{G/K,d}$, we set
\[
u_c\cdot u_d:=\phi_d(u_c)u_d.
\]
Now checking associativity and commutativity are exactly as in Proposition \ref{ghost algebra prop}.
\end{proof}

\subsection{Lifting $\AA_{G/K}^{full}$ to $G$-equivariant operators}\label{section lifting ops full}  Just as in Section \ref{sec lifting ops}, we may realize $\AA_{G/K}^{full}$ as an algebra of operators on $G/K$.  For this observe that we have a surjection $\Ad:Z(G_0)\to\Aut(\g,\g_{\ol{0}})$.  We \textbf{assume} that there exists a torus $\G_m\sub Z(G_0)$ such that the map $\G_m\to\Aut(\g,\g_{\ol{0}})$ is surjective, and that $\Ad^{-1}(\pm1)\cap\G_m\sub K_0$.  

Now, given $u\in\AA_{G/K,c}$, we let $z\in\G_m$ be such that $\Ad(z)=c\in\Aut(\g,\g_{\ol{0}})$ and set:
\[
\tilde{u}:=R_{z}^*\circ(1\otimes u)\circ a^*.
\]
Then we have
\begin{thm}
	The map $u\mapsto\tilde{u}$ defines an injective morphism of algebras
\[
\AA_{G/K}^{full}\to\End_{G}( \Bbbk[G/K])
\]
such that for $\lambda\in\Lambda$ we have $\frac{\tilde{u}(f_{\lambda})}{f_{\lambda}}=(e^\lambda(z))HC(u)(-\lambda)$.
\end{thm}

\begin{proof}
	The proof works in the exact same way as that of \cref{thm lifting to ops}.
\end{proof}

\subsection{The algebra $\DD^{G,\bullet}(G/K)$}
We continue with the same assumption of \cref{section lifting ops full}, meaning that we have a torus $\G_m\sub Z(G_0)$ such that the map $\Ad:\G_m\to\Aut(\g,\g_{\ol{0}})$ is surjective and $\Ad^{-1}(\pm1)\cap\G_m\sub K_0$.  Because $\Ad^{-1}(1)=Z(G)$ and we may as well quotient by $Z(G)\cap K$, we can and will assume that $\Ad:\G_m\to\Aut(\g,\g_{\ol{0}})$ is an isomorphism.  Thus we only need to \textbf{assume} that $\Ad^{-1}(-1)\cap\G_m\sub K_0$.

Let $z\in \G_m$, and $u\in\AA_{G/K,\Ad(z)}$.  Then define
\[
D_u:=L_{z^{-1}}^*\circ R_{z}^*\circ(1\otimes u)\circ a^*: \Bbbk[G/K]\to \Bbbk[G/K]
\]
\begin{lemma}\label{lemma D_u diff op}
	The operator $D_u$ defines a differential operator on $G/K$ such that it is $\Ad(z^{-1})$-twisted equivariant, i.e.~
	\[
	vD_u-D_u\Ad(z^{-1})(v)=0
	\]
	for all $v\in\g$.
\end{lemma}

First we prove a lemma.

\begin{lemma}\label{diff_op_criteria}
	Let $X$ be a smooth affine supervariety of dimension $(m|n)$, and let $L$ be an operator on $\Bbbk[X]$.  Suppose that for every closed point $x\in X(\Bbbk)$, there exists $j>0$ such that for $N\geq j$ we have $L(\m_x^N)\sub\m_x^{N-j}$.  Then $L$ is a differential operator.
\end{lemma}

\begin{proof}
Let $f_1,\dots,f_M\in \Bbbk[X]$ where $M=N+n+1$, and without loss of generality assume that $X$ admits global odd coordinates $\xi_1,\dots,\xi_n$.  For $g\in \Bbbk[X]$ we have:
\[
   [f_1,\cdots,[f_M,L]\cdots](g)=\sum\limits_{J}\pm f_{J^c}L(f_Jg),
\]
   where $J\sub\{1,\dots,M\}$ is a subset and $J^c$ is the complement. For a closed point $x\in X(\Bbbk)$, we may assume that $f_i(x)=0$ for all $i$, and so we see that the above expression lies in $\m_x^{n+1}$.  Since $x$ is arbitrary, we may conclude from the following lemma, which is a consequence of Thm A.2 of \cite{masuoka_zubkov}.
\end{proof}

\begin{lemma}
	Let $X$ be a smooth affine supervariety of dimension $(m|n)$.  Then
	\[
	\bigcap\limits_{x\in X(\Bbbk)}\m_x^{n+1}=0.
	\] 
\end{lemma}

\begin{proof}[Proof of Lemma \ref{lemma D_u diff op}]
	The twisted equivariance is a straightforward check.  It remains to check that it defines a differential operator, and for this we use \cref{diff_op_criteria}.   First of all notice that since $D_u$ is $G_0$-equivariant, it suffices to prove that there exists $j\geq0$ such that for $N\geq j$ we have $D_u(\m_{eK}^N)\sub\m_{eK}^{N-j}$.
	
	For this we begin by noticing that $\operatorname{res}_{eK}D_{u}=u$; thus if we let $j$ be the degree of $u$, we have $D_u(\m_{ek}^N)\sub\m_{eK}$ for all $N>j$.  Now suppose that for $N\geq j$ we have $D_u(\m_{eK}^N)\not\sub\m_{eK}^{N-j}$.  Then there exists $f\in\m_{eK}^N$ such that $D_u(f)\in\m_{eK}^{r}$ for some $r<N-j$.  Since $G/K$ is a homogeneous space, it follows that there exists $v_1,\dots,v_r\in\g$ such that $v_1\cdots v_rD_u(f)(eK)\neq0$.  However by the invariance properties of $D_u$, it must follow that
	\[
	v_1\cdots v_rD_u(f)=D_u(\Ad(z^{-1})(v_1)\cdots \Ad(z^{-1})(v_r)f)(eK)\neq0.	
	\]
	However $\Ad(z^{-1})(v_1)\cdots \Ad(z^{-1})(v_r)f\in\m_{eK}^{N-r}$;  by assumption, $N-r>j$, so we must have $D_u(\Ad(z^{-1})(v_1)\cdots \Ad(z^{-1})(v_r)f)\in\m_{eK}$, a contradiction.  This completes the proof.	
\end{proof}

Write $\DD^{G,z}(G/K)$ for the $\Ad(z)$-twisted equivariant differential operators on $G/K$.  Let 
\[
\DD^{G,\bullet}(G/K):=\sum\limits_{z\in \G_m}\DD^{G,z}(G/K).
\]
Then $\DD^{G,\bullet}(G/K)$ is a subalgebra of $\DD(G/K)$.

\begin{thm}
	The map
	\[
	\Dist(G/K,eK)^{G_z}\to\DD^{G,z^{-1}}(G/K), \ \ \ u\mapsto D_u
	\]
	defines an isomorphism of vector spaces.
\end{thm}

\begin{proof}
	We already saw that $\operatorname{res}_{eK}D_u=u$, so it remains to show that $\operatorname{res}_{eK}:\DD^{G,z}(G/K)\to\Dist(G/K,eK)^{G_z}$ is injective.  The proof is almost identical to Prop. 3.4 of \cite{sherman2021ghost}, but we give it once again.
	
	Let $D\in\DD^{G_z}(G/K)$; then we have that
	\[
	a^*\circ D=\Ad(z^{-1})\otimes D\circ a^*.
	\]
	If $\operatorname{res}_{eK}(D)=0$, then $D(f)(eK)=0$ for all $f\in \Bbbk[G/K]$, or equivalently $a_{eK}^*D(f)(eK)=0$, where $a_{eK}:G\to G/K$ is the orbit map at $eK$.  But we have $a_{eK}=a\circ(\id_G\times i_{eK})$, so this says that
	\begin{eqnarray*}
	(\id_G\otimes i_{eK})\circ a^*(D(f))& = &(\Ad(z^{-1})^{*}\otimes \operatorname{res}_{eK}(D))(a^*(f))=0.
	\end{eqnarray*}
Thus $a_{eK}^*(D(f))=0$, which implies in turn that $D(f)=0$, so that $D=0$ as desired.	
\end{proof}
Note that we do not obtain an algebra map $\AA^{full}_{G/K}\to\DD^{G,\bullet}(G/K)$ because $K_{z}=K_{-z}$, while $\DD^{G,z}(G/K)\neq\DD^{G,-z}(G/K)$.

\subsection{Map from full ghost center of $\g$} Recall from Section 10 of \cite{sherman2021ghost} the full ghost center $\ZZ_{full}$ of $\UU\g$.  It is defined as follows: given $\phi_c\in\Aut(\g,\g_{\ol{0}})$, we let 
\[
\AA_{c}=\{u\in\UU\g: vu-(-1)^{\ol{u}\ol{v}}u\phi_c(v)=0\text{ for all }v\in\g\}.
\]
Then $\AA_c\AA_d\sub\AA_{cd}$, and we may define
\[
\ZZ_{full}:=\sum\limits_{c}\AA_c.
\]
The structure of this algebra was computed in \cite{sherman2021ghost}.  Now observe that we have a natural map
\[
\AA_{\Ad(z)}\to \DD^{G,z}(G/K).
\]
This induces an algebra homomorphism
\[
\ZZ_{full}\to \DD^{G,\bullet}(G/K).
\]
This map cannot be surjective; indeed, for $c\neq 1$, $HC(\AA_{c})$ was computed in \cite{sherman2021ghost}, and the lowest degree polynomial lying in $HC(\AA_c)$ is $\dim\g_{\ol{1}}/2$.  In our situation, if $c=-1$ we obtain $HC(\DD^{G,-1}(G/K)=HC(\ZZ_{G/K})$, and for $c\neq \pm1$ we know that $HC(\AA_{G/K,c})$ contains a polynomial of degree $\dim\k_{\ol{1}}/2$.

\section{Tools for Computing $HC(\AA_{G/K})$}

In this section we offer a few tools which can help to understand $HC(\AA_{G/K})$ in some cases.  However we note that they are far from strong enough for determining $HC(\AA_{G/K})$.  

\subsection{Reduction of pair} Let $(\g,\k)$ be a supersymmetric pair with an Iwasawa decomposition $\g=\k\oplus\a\oplus\n$.  Let $\g(\theta,\a)$ denote the Lie superalgebra generated by $\a$ and $\p_{\ol{1}}$. Then $\theta$ will induce an involution, which we continue to denote by $\theta$ on $\g(\theta,\a)$, whose fixed points we call $\k(\theta,\a)$, and $(-1)$-eigenspace we write as $\p(\theta,\a)$.  It is clear that $\p(\theta,\a)_{\ol{1}}=\p_{\ol{1}}$.

Clearly $\a\sub\g(\theta,\a)$ will continue to be a Cartan subspace, and the Iwasawa decomposition $\g(\theta,\a)=\k(\theta,\a)\oplus\a\oplus\n(\theta,\a)$ continues to hold, where $\n(\theta,\a)=\n\cap\g(\theta,\a)$.  It is not hard to check that $(\n(\theta,\a))_{\ol{1}}=\n_{\ol{1}}$.  

We may consider $\k(\theta,\a)'$; then we have a natural map
\[
\iota:\UU\k(\theta,\a)'/\UU\k(\theta,\a)'\k(\theta,\a)'_{\ol{0}}\to \UU\k'/\UU\k'\k_{\ol{0}}
\]
induced by the natural inclusion $\k(\theta,\a)'\to\k'$.
\begin{lemma}\label{lemma reduction trick}
	The map $\iota$ is an isomorphism, and for some $c\in \Bbbk^\times$ we have
	\[
	\iota(v_{\k(\theta,\a)})=cv_{\k'}.
	\]
\end{lemma}
\begin{proof}
	The fact that $\iota$ is an isomorphism follows from the fact that $\k(\theta,\a)'_{\ol{1}}=\k_{\ol{1}}'$.  It is clear that $\iota$ is $\k(\theta,\a)'$-equivariant, and so since $v_{\k'}$ is annihilated by $\k(\theta,\a)'$, by uniqueness (Cor. 6.2 of \cite{sherman2021ghost}) it is necessarily equal to $\iota(v_{\k(\theta,\a)})$ up to a nonzero scalar.
\end{proof}

\begin{cor}\label{cor HC v_k' comp reduction}
	There exists a nonzero scalar $c$ such that 
	\[
	HC(v_{\k(\theta,\a)})=cHC(v_{\k'}).
	\]
\end{cor}

	The use of Corollary \ref{cor HC v_k' comp reduction} is that the pair $(\g(\theta,\a),\k(\theta,\a))$ is sometimes simpler than the pair $(\g,\k)$; however generally the pairs are the same, perhaps up to a central extension.  In the below table we list the three cases where we truly get a simplification; note that we add the center to $\g(\theta,\a)$ to simplify matters.

\renewcommand{\arraystretch}{2}
\begin{tabular}{|c|c|}
	\hline 
	$(\g,\k)$ & $(\g(\theta,\a),\k(\theta,\a))$ \\
	\hline
 $(\g\l(m|n)$, $\g\l(m-r|n)\times\g\l(r))$, \ $2r\leq m$ & $(\g\l(2r|n),\g\l(r|n)\times\g\l(r))$  \\
	\hline
	$(\o\s\p(m|2n)$,$\o\s\p(m-r|2n)\times\s\o(r))$, \ $2r\leq m$& $(\o\s\p(2r|2n),\o\s\p(r|2n)\times\s\o(r))$  \\
	\hline
	$(\o\s\p(m|2n)$,$\o\s\p(m|2n-2r)\times\s\p(2r))$, \ $2r\leq n$ & $(\o\s\p(m|4r),\o\s\p(m|2r)\times\o\s\p(2r))$  \\
	\hline
\end{tabular}

\subsection{Vanishing criteria}
\begin{prop}\label{iso_root_vanishing_general}
	Let $\alpha$ be a simple isotropic root such that $\theta\alpha\neq\alpha$ and $\theta\alpha+\alpha\notin\Delta$.  Then $h_{\alpha}|HC(D)$ for all $D\in\AA$.
\end{prop}

\begin{proof}
Suppose that $\lambda\in\Lambda^+$ and that $(\lambda,\alpha)=0$ and $V(\lambda)$ contains a copy of $I_{\k'}( \Bbbk)$.  Then we may write $V(\lambda)=I_{\k'}( \Bbbk)\oplus M$ for some $\k'$-module $M$, and accordingly $f_{\lambda}=g+h$.  

Define $e_{-\ol{\alpha}}:=e_{-\alpha}-\theta e_{-\alpha}$.  Then $e_{-\alpha}f_{\lambda}=0$ by sphericity, so that $e_{-\ol{\alpha}}f_{\lambda}=0$; in particular $e_{-\ol{\alpha}}g=0$.  Under our conditions, $[e_{-\ol{\alpha}},e_{-\ol{\alpha}}]=0$, and so by projectivity we must have $g=e_{-\ol{\alpha}}g'$ for some $g'\in I_{\k'}( \Bbbk)$.  However since $g$ generates $I_{\k'}( \Bbbk)$ this is impossible, so we obtain a contradiction.
\end{proof}

\begin{remark}
	An interesting (albeit unfortunate) caveat to \cref{iso_root_vanishing_general} is that the hypothesis almost never holds(!).  The only supersymmetric pairs for which it does hold are $(\g\l(m|n),\g\l(m-r|n-s)\times\g\l(r|s))$, and for certain isotropic roots of the pairs $(\o\s\p(m|2n),\o\s\p(r|2s)\times\o\s\p(m-r|2n-2s))$.  
	
	This is in great contrast to the classical, even setting where it is always true that $\alpha+\theta\alpha$ is not a root.  
\end{remark}

\subsection{A trick for certain pairs}

For each of the Lie superalgebras we consider the following decompositions $\g_{\ol{0}}=\bigoplus\limits_i\g_{\ol{0}}^i$: if $\g_{\ol{0}}$ is semisimple, let $\g_{\ol{0}}^i$ be its simple components; if $\g=\o\s\p(2|2n)$, let $\g_{\ol{0}}^1=\s\o(2)$ and $\g_{\ol{0}}^2=\s\p(2n)$, and finally if $\g=\g\l(m|n)$ then let $\g_{\ol{0}}^1=\g\l(m)$ and $\g_{\ol{0}}^2=\g\l(n)$.

Then in each case we obtain a corresponding decomposition $\h=\bigoplus\limits_i\h_i$ of the Cartan subalgebra, where $\h_i\sub\g_{\ol{0}}^i$ will be a Cartan subalgebra; correspondingly we obtain decompositions of $\h^*$.  Now write $\ZZ_{\g},\ZZ_{\g^i_{\ol{0}}}$ for the centers of $\UU\g,\UU\g^i_{\ol{0}}$ respectively.  Then we have identifications $HC(\ZZ_{\g})=S(\h)^{\mathcal{W}_\rho}$ and $HC(\ZZ_{\g_{\ol{0}}^i})\sub S(\h^i)^{W^i_{\rho_i}}$, where $\mathcal{W}$ is the Weyl groupoid of $\g$, $W_i$ is the Weyl group of $\g_{\ol{0}}^i$, and $\rho$, resp. $\rho_i$ is the Weyl vector of $\g$, resp. $\g^i_{\ol{0}}$.  Here the subscripts indicate that we are taking the $\rho$ or $\rho_i$ shifted actions.

Now suppose that $(\g,\k)$ is a supersymmetric pair satisfying the Iwasawa decomposition and which has that $\a\sub\h^i$ for some $i$.  Then we have the following commutative diagram: 
\[
\xymatrix{\ZZ_{\g}\ar[d]^{HC} & & \ZZ_{\g_{\ol{0}}^i}\ar[dd]^{HC} \\ S(\h)^{\mathcal{W}_{\rho}} \ar[d]^p \ar[r] & S(\h)^{\mathcal{W}}\ar[d]^p & \\ S(\h_i)^{W^i_{p(\rho)}} \ar[d]^q \ar[r] & S(\h_i)^{W^i}\ar[d]^q & S(\h_i)^{W^i_{\rho_i}} \ar[l]\ar[d]^q \\
	S(\a)^{W^{lit}_{qp(\rho)}} \ar[r] & S(\a)^{W^{lit}} & S(\a)^{W^{lit}_{q(\rho^i)}}\ar[l] }.
\]
The maps $p$ and $q$ are projection maps onto subspaces; all horizontal arrows are pullbacks under translation by the appropriate vector, and are obviously isomoprhisms.  Finally, $W^{lit}$ is the little Weyl group of the supersymmetric space.  By Helgason's theorem (\cite{helgason1992some}), the composite map $\ZZ_{\g_{\ol{0}}^i}\to S(\h_i)^{W^i_{\rho_i}}\to S(\a)^{W^{lit}_{q(\rho^i)}}$ is a surjective map for these pairs.  It follows that if the map $p:\ZZ_{\g}\to S(\h_i)^{W^i_{p(\rho)}}$ is surjective, then so is the map $qp:\ZZ_{\g}\to S(\a)^{W^{lit}_{qp(\rho)}}$, and in particular we would have that the natural map
\[
\ZZ_{\g}\to\ZZ_{G/K}
\]
is surjective in these cases.  The following proposition explains when this occurs.  

\begin{prop}\label{prop center surjective}
	The map $\ZZ_{\g}\to S(\h_i)^{W^i_{p(\rho)}}$ is surjective if and only if we are in one of the following cases:
		\begin{itemize}
			\item $\g=\g\l(m|n),\o\s\p(2m+1|2n),\d(1|2;\alpha)$, $\g(1|2)$; 
			\item $\g=\o\s\p(2m|2n)$ with $\g_{\ol{0}}^i=\s\p(2n)$; 
			\item $\g=\a\b(1|3)$ with $\g_{\ol{0}}^i=\s\l(2)$.
		\end{itemize}
\end{prop}
\begin{proof}
	The proof is done case by case; clearly our question is equivalent to when the map $p:S(\h)^{\mathcal{W}}\to S(\h_1)^{W^1}$ is surjective.  In \cite{sergeev1999invariant}, generators for $S(\h)^{\mathcal{W}}$ were described for each of the Lie superalgebras we consider. 

	\begin{enumerate}
		\item[(i)] $\g=\g\l(m|n)$: generators for $S(\h)^{\mathcal{W}}$
		\[
		\sum\limits_i \epsilon_i^k-\sum\limits_j\delta_j^k, \ \ \ k\in\N.
		\]
		Generators for $\g\l(m)$:
		\[
		\sum\limits_i\epsilon_i^k, \ \ \ k\in\N.
		\]
		\item[(ii)] $\g=\o\s\p(2m+1|2n)$: generators for $S(\h)^{\mathcal{W}}$:
		\[
		\sum\limits_i\epsilon_i^{2k}-\sum\limits_j\delta_j^{2k}.
		\]
		Generators for $\s\o(2m+1)$:
		\[
		\sum\limits_{i}\epsilon_i^{2k}.
		\]
		Generators for $\s\p(2n)$:
		\[
		\sum\limits_i\delta_i^{2k}.
		\]
		\item[(iii)] $\g=\o\s\p(2m|2n)$: every element of $S(\h)^{\mathcal{W}}$ is of the form:
		\[
		f_o+f_1\epsilon_1\cdots\epsilon_m\prod\limits_{i,j}(\epsilon_i^2-\delta_j^2),
		\]
		where $f_1\in S(\h)^{W}$ and $f_0$ lies in the subalgebra generated by the polynomials
		\[
		\sum\limits_i\epsilon_i^{2k}-\sum\limits_j\delta_j^{2k}.
		\]
		Thus we obtain surjectivity on the component with $\g_{\ol{0}}^i=\s\p(2n)$, but the component with $\g_{\ol{0}}^i=\s\o(2m)$ has one generator given by $\epsilon_1\cdots\epsilon_m$, and this will not be in the image.

		\item[(iv)] $\g=\mathfrak{d}(1|2;\alpha)$: let $\lambda_1=-(1+\alpha)$, $\lambda_2=1$, and $\lambda_3=\alpha$.  Then $S(\h)^{\mathcal{W}}$ contains the polynomial 
		\[
		\frac{1}{\lambda_1}\epsilon_1^2+\frac{1}{\lambda_2}\epsilon_2^2+\frac{1}{\lambda_3}\epsilon_3^2,
		\]
		so its restriction to any component $\h_i$ will generate $S(\h_i)^{W^i}$.
		
		\item[(v)] $\g=\g(1|2)$: every element of $S(\h)^{\mathcal{W}}$ is of the form
		\[
		f_0+\prod\limits_{1\leq i\leq 3}(\delta_i^2-\epsilon_i^2)f_1,
		\]
		where $f_0\in \Bbbk[3\delta^2-2(\epsilon_1^2+\epsilon_2^2+\epsilon_3^2)]$, and $f_1\in S(\h)^{W}$.  If $\g_{\ol{0}}^i=\s\l(2)$ then we clearly obtain a surjective map, and if $\g_{\ol{0}}^i=\g(2)$, then generators for $S(\h^i)^{W^i}$ are given by
		\[
		\epsilon_1^2+\epsilon_2^2+\epsilon_3^2, \ \ \ \epsilon_1^2\epsilon^2\epsilon^2.
		\]
		We see these are in the image of the restriction, so we again get surjectivity.
		
		\item[(vi)] $\g=\a\b(1|3)$: every element of $S(\h)^{\mathcal{W}}$ is of the form
		\[
		f_0+f_1\prod(\delta\pm\epsilon_1\pm\epsilon_2\pm\epsilon_3)
		\]
		where $f_1\in S(\h)^{W}$ and $f_0\in\C[L_2,L_6]$ where $L_2=\delta^2-2(\epsilon_1^2+\epsilon^2+\epsilon_3^2)$, and $L_6$ is a homogeneous degree 6 polynomial.  It follows that we get surjectivity for the component with $\g_{\ol{0}}^i=\s\l(2)$, while for the component with $\g_{\ol{0}}^i=\s\o(7)$ we do not get the degree 4 generator $\epsilon_1^4+\epsilon_2^4+\epsilon_3^4$, and thus we do not get surjectivity onto this component.
		
	\end{enumerate}
\end{proof}

We obtain the following application of \cref{prop center surjective}.  

\begin{thm}\label{thm even pairs}
	Suppose that $(\g,\k)$ is one of the following pairs:
	\[
	(\g\l(m|n),\g\l(m-r|n)\times\g\l(r)), \ m\geq 2r, \ \ \ (\o\s\p(2m+1|2n),\o\s\p(2m+1-r|2n)\times\s\o(r)), \ 2m+1\geq 2r,
	\]
	\[
	(\o\s\p(m|2n),\o\s\p(m|2n-2s)\times\s\p(2s)), \ 2n\geq 4s, \ \ \ (\g(1|2),\mathfrak{d}(1|2,3)).
	\]
	Then $\ZZ_{\g}\to\ZZ_{G/K}$ is surjective, and 
	\[
	HC(\AA_{G/K})=S(\a)^{W^{lit}_{\ol{\rho}}}\cdot HC(v_{\k'}),
	\]
	where $S(\a)^{W^{lit}_{\ol{\rho}}}$ is the invariance of $S(\a)$ under the $\ol{\rho}$-shifted action of the little Weyl group $W^{lit}$, and $\ol{\rho}$ is the restriction of $\rho$ to $\a$.
\end{thm}
\begin{proof}
	The statement that $\ZZ_{\g}\to\ZZ_{G/K}$ is surjective is a direct consequence of \cref{prop center surjective} and the diagram above it.  For the second statement, we have maps:
	\[
	\xymatrix{HC(\ZZ_{\g}) \ar[dr]^{m} & \\
		\ZZ_{G_0/K_0} \ar[r]^{\omega} & HC(\AA_{G/K})}
	\]
	where $m(z)=HC(z)HC(v_{\k'})$, and $\omega(z)=HC(v_{\k'}z)$.  Of course $\omega$ is an isomorphism by construction; on the other hand, by \cref{cor highest order term}, $\deg HC(v_{\k'}z)=\deg HC(z)\cdot\dim \p_{\ol{1}}/2$.  Of course $m$ is also injective, and we have $\deg m(z)=\deg HC(v_{\k'})\deg HC(z)$. 
	
	Now by \cref{prop center surjective} and the diagram above it, the subspaces $\{z\in\ZZ_{G_0/K_0}:\deg HC(z)\leq r\}$ and $\{z\in\ZZ_{G/K}:\deg HC(z)\leq r\}$ have the same dimension.  Therefore the image of $m$ must be equal to the image of $\omega$, implying that $m$ is also surjective.
		
	By the diagram above \cref{prop center surjective} we have that $HC(\ZZ_{G/K})=S(\a)^{W^{lit}_{\ol{\rho}}}$, finishing the proof.
\end{proof}

\section{Rank 1 Computations}

In this section we compute $HC(\AA_{(\g,\k)})$ for all supersymmetric pairs that we consider of rank 1, i.e.~ in which $\dim\a=1$.  This includes the following pairs:
\[
(\o\s\p(1|2)\times\o\s\p(1|2),\o\s\p(1|2)), \ \ \ (\s\l(1|1)\times\s\l(1|1),\s\l(1|1)),
\]
\[
(\g\l(m|n),\g\l(m-1|n)\times\g\l(1)), \ m\geq2, \ \ \ (\o\s\p(m|2n),\o\s\p(m-1|2n)), \  m\geq 2,
\]
\[
(\o\s\p(m|2n),\o\s\p(m|2n-2)\times\s\p(2)), \ n\geq 2.
\]
We include the first two pairs because they are also of rank one. It is possible that a general formula for $HC(\AA_{(\g,\k)})$ could be proven by reduction to the rank one case, in which case the rank one pairs will be especially important, and these diagonal ones may appear in the process.  Notice that $(\s\l(1|2),\o\s\p(1|2))\cong(\o\s\p(2|2),\o\s\p(1|2))$, so it is also included.

In the following, for the diagonal pairs $(\g\times\g,\g)$ we present the answer in $S(\mathfrak{h})$ for a Cartan subalgebra $\mathfrak{h}\sub\g$, since this is equivalent.

\begin{thm} We have the following computation of $HC(\AA_{(\g,\k)})$:
	\begin{enumerate}
		\item[(i)] $(\o\s\p(1|2)\times\o\s\p(1|2),\o\s\p(1|2))$: let $\delta(h_{\delta})=1$; then
		\[
		HC(\AA_{(\g,\k)})=(h_{\delta}+\frac{1}{2}) \Bbbk[\mathfrak{h}]^{W_{\rho}}.
		\]
		\item[(ii)] $(\s\l(1|1)\times\s\l(1|1),\s\l(1|1))$: let $h_{\alpha}$ be a coroot of the odd isotropic root; then
		\[
		HC(\AA_{(\g,\k)})=h_{\alpha} \Bbbk[\mathfrak{h}].
		\]
		\item[(iii)] $(\g\l(m|n),\g\l(m-1|n)\times\g\l(1))$: let $t=\frac{1}{2}h_{\epsilon_1-\epsilon_m}$; then
		\[
		HC(\AA_{(\g,\k)})= \Bbbk[t(t-n+m-1)]\langle t(t-1)\cdots(t-(n-1))\rangle.
		\]
		\item[(iv)] $(\o\s\p(2|2n),\o\s\p(1|2n))$: let $t=h_{\epsilon_1}$, where $\epsilon_1(h_{\epsilon_1})=1$; then 
		\[
		HC(\AA_{(\g,\k)})=\{p\in \Bbbk[t]:p(n+r)=(-1)^{r}p(n-r):1\leq r\leq n\},
		\]
		or more explicitly:
		\[
		\Bbbk[t(t-2n)]\langle (t-1)(t-3)\cdots(t-(2n-1)),t(t-2)\cdots(t-2n)\rangle.
		\]
		\item[(v)] $(\o\s\p(m|2n),\o\s\p(m-1|2n))$, $m\geq3$: let $t=h_{\epsilon_1}$; then
		\[
		HC(\AA_{(\g,\k)})= \Bbbk[t(t-2n+m-2)]\langle(t-1)(t-3)\cdots(t-(2n-1))\rangle.
		\]
		\item[(vi)] $(\o\s\p(m|2n),\o\s\p(m|2n-2)\times\s\p(2))$, $n\geq 2$: let $t\in\a$ be such that $(\delta_1+\delta_2)(t)=1$; then
		\[
		HC(\AA_{(\g,\k)})= \Bbbk[t(t+2n-m-1)]\langle (t+1)t(t-1)\cdots(t-(m-2))\rangle.
		\]
	\end{enumerate}
\end{thm}
The proofs of (i) and (ii) follow from \cite{gorelik2000ghost}.  The proofs for the rest of the pairs will occupy the rest of the section; we will go through each case individually, using different techniques.  
%
In addition, we have the following computation which we prove and subsumes part (iii):
\begin{thm}
	For the pair $(\g\l(m|n),\g\l(m-r|n)\times\g\l(r))$, set $t_i=\frac{1}{2}h_{\epsilon_i-\epsilon_{r+i}}$ (see \cref{section computation A pair} for more on the setup).  Then we have:
	\[
	HC(\AA_{(\g,\k)})=S(\a)^{W^{lit}_{\ol{\rho}}}\prod\limits_{\substack{1\leq i\leq r\\ 1\leq j\leq n}}(t_i-n+r-i+j).
	\]
\end{thm}

See \cref{sec_intro_conj} for some remarks on the above computations, and a conjecture for the image $HC(\AA_{(\g,\k)})$ in all interlaced cases.

\subsection{$(\o\s\p(2|2n),\o\s\p(1|2n))$}

We present $\g=\o\s\p(2|2n)$ as
\[
\begin{bmatrix}t & 0 & \alpha^t & -\beta^t\\
	0 & -t & \sigma^t & -\tau^t\\\tau & \beta & m & p\\\sigma & \alpha & q & -m^t\end{bmatrix}.
\]
The involution realizing this pair can be given by conjugation by
\[
\begin{bmatrix}0 & 1 &\\1 & 0 &\\ &&\id\end{bmatrix}.
\]
The Cartan subspace $\a$ is given by the span of the matrix $t$ above, and thus $\a^*$ is spanned by a single weight $\epsilon$ where $\epsilon(t)=1$.  We thus identify $S(\a)= \Bbbk[t]$ and $\a^*$ with $\Bbbk$ via $a\epsilon\leftrightarrow a$.  We make the following choices:
\[
\n=\begin{bmatrix}0 & 0 & u^t & -r^t\\
	0 & 0 & 0 & 0\\0 & r & 0 & o\\0 & u & 0& 0\end{bmatrix}
\]
and
\[
\n^-=\begin{bmatrix}0 & 0 & 0 & 0\\
	0 & 0 & v^t & -s^t\\s & 0 & 0 & 0\\v & 0 & 0 & 0\end{bmatrix}.
\]
Let $\omega=\sum\limits_i s_iv_i\in\UU\n^-$.  One can show that for $1\leq r\leq n$ we have that $\omega^rf_{n+r}$ is non-zero of highest weight $(n-r)\epsilon$.  Now suppose that we write $\omega^r=k+p+m$ where $m\in\n\UU\g$, $p\in S(\a)$, and $k\in S(\a)(\UU\k)^+$.  Then clearly we have
\[
(\omega^rf_{n+r})(eK)=p(n+r).
\]
Thus
\[
f_{n-r}=\frac{1}{p(n+r)}\omega^rf_{n+r}.
\]
On the other hand, let $\phi\in\Aut(\g,\g_{\ol{0}})$ denote the interlacing automorphism of this pair given by $(-i)\id_{\g_{-1}}\oplus\id_{\g_0}\oplus i\id_{\g_{1}}$.  Then as discussed previously, $\phi$ fixes $\a$ pointwise, preserves $\n$, and sends $\k$ to $\k'$.  Thus $\phi(\omega^r)=k'+p+m'$ where $m'\in\n\UU\g$ and $k'\in S(\a)(\UU\k')^+$.  On the other hand $\omega\in S^2\n^-$, thus $\phi(\omega^r)=(-1)^r\omega^r$.  Therefore
\[
f_{n-r}=\frac{(-1)^r}{p(n+r)}\phi(\omega^r)f_{n+r}=k'f_{n+r}+(-1)^rf_{n+r}.
\]
It follows that $HC(D)(n+r)=(-1)^{r}HC(D)(n-r)$ for all $1\leq r\leq n$.  Let $\PP_{n}\sub \Bbbk[t]$ denote the collection of polynomials given by
\[
\PP_n:=\{p\in \Bbbk[t]:p(n+r)=(-1)^{r}p(n-r):1\leq r\leq n\}.
\]  
Then we have shown that $HC(\AA)\sub\PP_n$.  One can check that $\PP_n$ is of codimension $n$ in $ \Bbbk[t]$.  Here $\ZZ_{G_0/K_0}=S(\a)$ and $\deg HC(v_{\k'})=n$, so that $HC(\AA)$ is also of codimension $n$ in $ \Bbbk[t]$, so we obtain that:
	\[
	HC(\AA)=\PP_n:=\{p\in \Bbbk[t]:p(n+r)=(-1)^{r}p(n-r):1\leq r\leq n\}.
	\]
	Explicitly, we may write this as
	\[
	\Bbbk[t(t-2n)]\langle (t-1)(t-3)\cdots(t-(2m-1)),t(t-2)\cdots(t-2n)\rangle,
	\]
	i.e.~ the $ \Bbbk[t(t-2n)]$-module generated by the two written polynomials. Here $t(t-2n)=HC(\Omega)$, where $\Omega$ is the Casimir of this space.

\subsection{$(\o\s\p(m|2n),\o\s\p(m-1|2n))$, $m\geq3$}  For this pair, we use \cref{lemma reduction trick} and our computation for $(\o\s\p(2|2n),\o\s\p(1|2n))$, which tells us that 
\[
HC(v_{\k'})=(t-1)\cdots(t-(2n-1)).
\]
Thus
\[
HC(\AA_{(\g,\k)})= \Bbbk[t(t-2n+m-2)]\langle(t-1)(t-3)\cdots(t-(2n-1))\rangle.
\]
Here $t(t-2n+m-2)=HC(\Omega)$, where $\Omega$ is the Casimir.

\subsection{$(\o\s\p(m|2n),\o\s\p(m|2n-2)\times\s\p(2))$}  In this case the involution on the root system may be given by $\delta_1\leftrightarrow-\delta_{2}$, and thus $\a^*$ is spanned by $\delta_1+\delta_2$.  With respect to an Iwaswa Borel subalgebra, the simple module $L(\delta_1+\delta_2)$ of highest weight $\delta_1+\delta_2$ is a quotient of $S^2\Bbbk^{m|2n}$, where $\Bbbk^{m|2n}$ is the standard module for $\o\s\p(m|2n)$.  One can check directly that it admits a $\k$-coinvariant, and thus appears in $\Bbbk[G/K]$; it follows that $\Lambda^+=\{n(\delta_1+\delta_2)\}_{n\in\N}$.  

Now we use \cref{lemma reduction trick} to reduce the computation of $HC(v_{\k'})$ to the case of \linebreak $(\o\s\p(m|4),\o\s\p(m|2)\times\s\p(2))$.  In this case, $c(\delta_1+\delta_2)$ has the same central character as $d(\delta_1+\delta_2)$ if and only if $d=c$ or $d=m-3-c$; the map $c\mapsto m-3-c$ is exactly the $\rho$-shifted action of the little Weyl group.  By \cref{thm even pairs}, $HC(\ZZ_{(\g,\k)})$ is equal to the invariants of the little Weyl group, or equivalently the polynomial algebra in the Casimir.  By \cref{ghost algebra prop}, $HC(v_{\k'})^2\in HC(\ZZ_{(\g,\k)})$; it follows that the zeros of $HC(v_{\k'})$ must be stable under the $\rho$-shifted action of the little Weyl group.

Now assume that $m\geq 2$; then by the above considerations the simple modules of highest weight $c(\delta_1+\delta_2)$ must appear for $0\leq c\leq m-2$.  All these simple modules appear as subquotients of $(\Bbbk^{m|4})^{\otimes r}$ for $r\leq 2(m-2)$; we claim that $I_{\k'}(\Bbbk)$, the injective indecomposable on the trivial module for $\k'$, does not appear in any such power.  Since $\Bbbk^{m|4}$ restricts to $\k'$ as $\Bbbk^{m|2}\oplus\Bbbk^{0|2}$, we in fact want to show that $I_{\k'}(\Bbbk)$ does not appear in $(\Bbbk^{m|2})^{\otimes r}$ for $r\leq 2(m-2)$.

To prove this we use results on the structure of the tensor powers of $\Bbbk^{m|2n}$ as a module over $OSp(m|2n)$.  There are two injective indecomposable $OSp(m|2n)$-modules $I$ whose restriction to $SOSp(m|2n)$ is given by $I_{SOSp(m|2n)}(\C)$, which are labeled in \cite{ehrig2022osp} as $P(0,+)$ and $P(0,-)$. Thus we want to determine the smallest $r$ such that $P(0,\pm)$ appear in $(\Bbbk^{m|2n})^{\otimes r}$.  By Theorem 12.2 of \cite{ehrig2022osp}, $P(0,-)$ first appears for the first time in $(\Bbbk^{m|2n})^{\otimes m(2n+1)}$.  When $n=1$ we get $3m$ which is indeed larger than $2m-4$. 

It remains to deal with $P(0,+)$.  By Theorem 7.3 of \cite{comes2017thick}, the indecomposable summands $R(\lambda)$ of $(\Bbbk^{m|2n})^{\otimes r}$ are parametrized by certain partitions $\lambda$ of size at most $r$.  Using Lemma 7.16 \cite{comes2017thick} and the rules of Section 14 \cite{ehrig2022osp} we have that the projective cover of the trivial module is given by the indecomposable representation $R((m-1)^{2n})$.  By Corollary 7.14 of \cite{comes2017thick}, $P(0,+)$ occurs as a direct summand in $V^{\otimes 2(m-1)n}$ and not in any smaller tensor power.  If $n=1$, this means that it does not appear as a direct summand in $(\Bbbk^{m|2n})^{\otimes k}$ for $k<2(m-1)$; since $2m-4<2m-2$, we have obtained a sufficient bound.

By \cref{lemma branching} it follows that $HC(v_{\k'})$ must vanish on all such $c$; and thus by the above-stated invariance property of its zeroes, we see that $HC(v_{\k'})$ vanishes at $c=-1,0,\dots,m-2$; by degree considerations we obtain that (up to a scalar)
\[
HC(v_{\k'})=(t+1)t(t-1)\cdots(t-(m-2)),
\]
where $t=\frac{1}{2}h_{\delta_1+\delta_2}$.  

Now if $m=1$, then we have the pair $(\o\s\p(1|4),\o\s\p(1|2)\times\s\p(2))$; in this case we have (up to a scalar)
\[
HC(v_{\k'})=t+1.
\]
This can be checked directly, or one can use that it is a degree one polynomial whose square is invariant under the $\rho$-shifted action of $W$, which completely determines it up to scalar.

Now we return to the general case $(\o\s\p(m|2n),\o\s\p(m|2n-2)\times\s\p(2))$; the Casimir $\Omega$ has (up to scalar) $HC(\Omega)=t(t+2n-m-1)$.  From this we obtain:
\[
HC(\AA_{(\g,\k)})=\Bbbk[t(t+2n-m-1)]\langle (t+1)t(t-1)\cdots(t-(m-2))\rangle.
\]

\subsection{$(\g\l(m|2n),\g\l(m-r|n)\times\g\l(r))$, $m\geq 2r$}\label{section computation A pair}	Using \cref{lemma reduction trick} and \cref{thm even pairs}, we reduce the question to the computation of $HC(v_{\k'})$ when the pair is $(\g\l(2r|n),\g\l(r|n)\times\g\l(r))$.  We present this algebra as
\[
\begin{bmatrix}a & b& \alpha\\c & d & \beta\\\phi & \psi & a\end{bmatrix}.
\]
Then we have the involution $\theta$ which is the permutation $(1,r+1)(2,r+2)\cdots(r,2r)$.	The fixed subalgebra $\k$ is isomorphic to $\g\l(r|n)\times \g\l(r)$, explicitly presented as:
\[
\k=\begin{bmatrix}a & b& \gamma \\b & a & \gamma \\ \sigma & \sigma & b\end{bmatrix}
\]
where $\gamma$ is of size $r\times n$ and $\sigma$ of size $n\times r$.

The $(-1)$-eigenspace $\p$ is given by
\[
\p=\begin{bmatrix}t & u& \tau\\-u & -t & -\tau\\\eta & -\eta & 0\end{bmatrix}.
\]
Thus a Cartan subspace is given by
\[
\a=\begin{bmatrix}D & 0 & 0\\0 & -D & 0\\0 & 0 & 0\end{bmatrix}
\]
where $D$ is $r\times r$ and diagonal. 	Let
\[
\n^+=\n=\begin{bmatrix}0 & e & \alpha\\0 & 0 & 0\\0 & \psi & 0\end{bmatrix},
\]
with positive root system for $\g$ given by
\[
\epsilon_1-\epsilon_2,\dots,\epsilon_{r-1}-\epsilon_r,\epsilon_r-\delta_1,\dots,\delta_{n-1}-\delta_n,\delta_n-\epsilon_{2r},\dots,\epsilon_{r+2}-\epsilon_{r+1}.
\]

We see that for this choice of Borel, the Weyl vector $\rho$ restricts to $\a$ as:
\[
\rho|_{\a}=\sum\limits_{i=1}^{r}\left(\frac{2(r-i)-n+1}{2}\right)(\epsilon_i-\epsilon_{r+i}).
\]
Here $\a^*$ is spanned by weights of the form:
\[
a_1(\epsilon_1-\epsilon_{r+1})+a_2(\epsilon_2-\epsilon_{r+2})+\dots+a_{r}(\epsilon_r-\epsilon_{2r}).
\]
By Prop. 6.11 of \cite{sherman2021ghost}, we may write:
\[
v_{\k'}=\prod\limits_{i,j}(\phi_{ij}-\psi_{ij})\prod\limits_{i,j}(\alpha_{ij}-\beta_{ij}).
\]
Now $\psi_{ij}\in\n^+$ and for all $i,j$ and since $\psi_{ij}$ commutes with $\phi_{k\ell}$, the above expression has the same Harish-Chandra projection (up to a sign) as: 
\[
\prod\limits_{i,j}\phi_{ij}\prod\limits_{i,j}(\alpha_{ij}-\beta_{ij}).
\]
We additionally have $\alpha_{ij}+\beta_{ij}\in\k$, so up to scalar the above has the same Harish-Chandra projection as:
\[
\prod\limits_{i,j}\phi_{ij}\prod\limits_{i,j}\alpha_{ij}
\]
From here the proof works by deal with all terms with index 1, and then concluding by induction.  Namely, begin by writing the above as
\[
\left(\prod\limits_{(i,j)\neq(1,1)}\phi_{ij}\right)\phi_{11}\alpha_{11}\left(\prod\limits_{(i,j)\neq(1,1)}\alpha_{ij}\right).
\]
If we move $\alpha_{11}$ all the way to the left (which we do since it lies in $\n$), we obtain one term given by
\[
\left(\prod\limits_{(i,j)\neq(1,1)}\phi_{ij}\right)\left(\prod\limits_{(i,j)\neq(1,1)}\alpha_{ij}\right)(h_{\epsilon_1-\delta_1}+n-r),
\]
and the rest of the terms contain $[\alpha_{11},\phi_{ij}]$ for $(i,j)\neq(1,1)$.  Now this commutator is nonzero in two cases: in one case $j=1$, and we get the even root vector of weight $\delta_i-\delta_1$, which lies in $\k$.  Thus we move it all the way to the right; when we move it past root vectors $\phi_{kl}$, we get something nonzero if and only if $k=1$, in which case we get $\phi_{il}$ for $i\geq 1$; but then $\phi_{il}$ will appear twice, and so we obtain 0.  If we move it past $\alpha_{kl}$, we get something nonzero if and only if $\ell=i$, in which case we obtain $\alpha_{ki}$ for $k\geq 1$ which will appear twice, so we again get 0.  Thus these terms all vanish.

The other case is if $i=1$ in which case $[\alpha_{11},\phi_{ij}]$ is a an even root vector of weight $\epsilon_{1}-\epsilon_j$, which lies in $\n$.  Thus we want to move it all the way to the left; doing so, we pick up new terms only when moving it past $\phi_{kl}$ for $l=1$, in which case we obtain $\phi_{kj}$, which will then appear twice, so the term becomes 0.  Thus only our first term written above survives.  

One may now continue and apply the same argument after moving $\alpha_{12}$ all the way to the right in order to obtain
\[
\left(\prod\limits_{(i,j)\neq(1,1),(2,1)}\phi_{ij}\right)\left(\prod\limits_{(i,j)\neq(1,1),(1,2)}\alpha_{ij}\right)(h_{\epsilon_1-\delta_1}+n-r)(h_{\epsilon_1-\delta_2}+n-r-1).
\]
Continue in this way, we obtain
\[
\left(\prod\limits_{j\neq 1}\phi_{ij}\right)\left(\prod\limits_{i\neq 1}\alpha_{ij}\right)\prod\limits_{j=1}^n(h_{\epsilon_1-\delta_j}+n-r-j+1).
\]
Now we deal with the terms with roots of the form $\epsilon_i-\delta_1$ for $1<i\leq r$, starting with $i=1$ and moving up in $i$.  Working in the same fashion as above, and we obtain
\[
\left(\prod\limits_{i,j\neq 1}\phi_{ij}\right)\left(\prod\limits_{i,j\neq 1}\alpha_{ij}\right)\prod\limits_{j=1}^n(h_{\epsilon_1-\delta_j}+n-r-j+1)\prod\limits_{i=2}^{r}(h_{\epsilon_i-\delta_1}+n-r+i-1).
\]
Now we can conclude inductively to obtain:
\[
\prod\limits_{\substack{1\leq i\leq r\\ 1\leq j\leq n}}(h_{\epsilon_i-\delta_j}+n-r+i-j).
\]
Now set $t_i=\frac{1}{2}h_{\epsilon_i-\epsilon_{r+i}}$; in particular $(\epsilon_i-\epsilon_{r+i})(t_i)=1$.  Applying the antipode on $\a$ we have shown that (up to scalar):
\[	
HC(v_{\k'})=\prod\limits_{\substack{1\leq i\leq r\\ 1\leq j\leq n}}(t_i-n+r-i+j).
\]

\textsc{\footnotesize Dept. of Mathematics, Ben Gurion University, Beer-Sheva,	Israel} 

\textit{\footnotesize Email address:} \texttt{\footnotesize xandersherm@gmail.com}

\end{document}